\documentclass[preprint,10pt]{amsart}
\usepackage{mathtools}
\usepackage{kpfonts}
\usepackage{amsmath}
\usepackage{amsthm}
\usepackage{amssymb}
\usepackage{mathdots}
\usepackage{pst-node}
\usepackage{enumitem}
\usepackage{algpseudocode}
\usepackage{algorithmicx}
\usepackage[T1]{fontenc}
\usepackage{algorithm}

\newtheorem{theorem}{Theorem}[section]
\newtheorem{lemma}[theorem]{Lemma}
\newtheorem{proposition}[theorem]{Proposition}

\newtheorem{example}[theorem]{Example}

\theoremstyle{remark}
\newtheorem{remark}[theorem]{Remark}

\numberwithin{equation}{section}

\def\dim{\mathop{\rm dim}\nolimits}

\def\Rea{\mathop{\rm Re}\nolimits}
\def\Ima{\mathop{\rm Im}\nolimits}

\def\Ham{\mathop{\rm Ham}\nolimits}
\def\Skew{\mathop{\rm Skew}\nolimits}
\def\Sp{\mathop{\rm Sp}\nolimits}
\def\O{\mathop{\rm O}\nolimits}
\def\C{\mathop{\rm C}\nolimits}

\def\GL{\mathop{\rm GL}\nolimits}
\def\G{\mathop{\rm G}\nolimits}
\def\U{\mathop{\rm U}\nolimits}
\def\R{\mathop{\rm R}\nolimits}
\def\L{\mathop{\rm L}\nolimits}

\def\H{\mathop{\rm H}\nolimits}

\def\Her{\mathop{\rm Her}\nolimits}

\parskip=\smallskipamount

\makeatletter
\renewcommand*\env@matrix[1][*\c@MaxMatrixCols c]{%
  \hskip -\arraycolsep
  \let\@ifnextchar\new@ifnextchar
  \array{#1}}
\makeatother

\begin{document}

\title[Centralizers]{Centralizers of elements in the orthogonal, symplectic and unitary groups}

\author{Tadej Star\v{c}i\v{c}}
\email{tadej.starcic@pef.uni-lj.si}

\address{Faculty of Education, University of Ljubljana, Kardeljeva Plo\v{s}\v{c}ad 16, 1000 Lju\-blja\-na, Slovenia}
\address{Institute of Mathematics, Physics and Mechanics, Jadranska 19, 1000 Ljubljana, Slovenia}


\begin{abstract} 
We explicitly compute and parametrize the centralizers of elements, encompassing dege\-ne\-rate and unipotent ones, in the orthogonal, symplectic, and unitary gro\-ups over the real and complex fields, including the indefinite cases. 
The same approach simultaneously yields the isotropy groups of $H$-skew-symme\-tric, Hamiltonian, and $H$-Hermitian matrices under the adjoint representations of the corresponding classical groups. 
The resulting groups are conjugate to subgroups of nonsingular block matrices whose blocks exhibit a rectangular block Toeplitz structure.
\end{abstract}

\date{September 27, 2026}

\maketitle


\medskip


%
%
%
%
\section{Introduction} \label{Sec1}

The study of classical matrix groups and their actions on matrix spaces natura\-lly leads to questions about conjugacy, centralizers, and isotropy groups. In this paper, we develop explicit and constructive methods for determining centralizers and isotry groups in several classical matrix groups and structured spaces.

Let $\mathbb{F}$ be either the field $\mathbb{C}$ or $\mathbb{R}$, let $\mathbb{F}^{n\times n}$ denote the space of $n\times n$ matrices over $\mathbb F$, and
let $\GL_n(\mathbb{F})
$ denote the group of nonsingular $n\times n$ matrices.
We write $I_n$ for the $n\times n$ identity matrix, for $p+q=n$ set $I_{p,q}:=I_p\oplus -I_q$, and $J_n:=\begin{bsmallmatrix}
0 & I_n\\
-I_n & 0
\end{bsmallmatrix}$. 
The \emph{orthogonal group} associated with $I_{p,q}$ and the standard \emph{symplectic group} are
\[ 
\O_{p,q}(\mathbb{F}):=\{R\in \GL_n(\mathbb{F}) \mid R^{T}I_{p,q}^{} R^{}= I_{p,q}\}, \quad \Sp_{2n}(\mathbb{F}):=\{R\in \GL_{2n}(\mathbb{F}) \mid R^{T}J_{n}^{} R^{}= J_{n}\},
\]
and over $\mathbb{C}$, the \emph{unitary group} associated with $I_{p,q}$ is:
\[
\U(p,q):=\{R\in \GL_n(\mathbb{C}) \mid R^{*}I_{p,q}^{} R^{} =I_{p,q} \}.
\]
The groups $\O_{p,q}(\mathbb{F})$ and $\U(p,q)$ contain, as special cases, the standard orthogonal group $\O_n (\mathbb{F})=\O_{n,0}(\mathbb{F})$ and the standard unitary group $\U(n)=\U(n,0)$.

For a matrix group $\G$ equal to one of $\O_{p,q}(\mathbb{F})$, $\Sp_{2n}(\mathbb{F})$ or $\U (p,q)$, consider its natural $\G$-\emph{con\-ju\-gation ($\G$-similarity)} action:
\begin{align}\label{ca} 
\G \times \G \to \G,\qquad (Q,R)\mapsto Q^{-1}RQ.
\end{align}
The corresponding \emph{centralizer}, or stabilizer, at an element $R\in\G$ is
\begin{align}\label{CGR}
\C_{\G}(R):=\{Q\in \G \mid Q^{-1}R Q=R\}. 
\end{align}
Centralizers play a fundamental role in the study of conjugacy classes; see, e.g., 
\cite {BurCush},\cite{Seitz}, \cite[Ch. E]{Spring}, \cite{Wall}. In particular, the conjugacy class of $R$ is naturally identified with the homogeneous space$\G/\C_{\G}(R)$.

There is a unique multiplicative Jordan-Chevalley decomposition of $R=US=SU$ with $U\in \G$ unipotent and $S\in \G$ semisimple (diagonalizable over $\mathbb{C}$), and $\C_{\G}(R)=\C_{\G}(S)\cap \C_{\G}(U)$. 
The centralizer $\C_{\G}(S)$ follows directly from the spectral decomposition of $S$, while the structure of the Levi factor of $\C_{\G}(U)$ is also known and determined by the Jordan structure of $U$, see, e.g., \cite[Theorem 3.1]{Seitz} and \cite[Ch. E.IV.2.25]{Spring}. However, these results
do not suffice to recover the entire centralizer $\C_{\G}(R)$
in the degenerate cases.
In particular, they do not provide an explicit description of the unipotent radical $R_u(\C_{\G}(U))$ itself, nor of the unipotent radical $R_u(\C_{\G}(R))$, its explicit realization inside $\C_{\G}(R)$, and its interaction with the reductive part of $\C_{\G}(R)$.

In this paper, we explicitly determine the centralizer $\C_{\G}(R)$ for arbitrary $R\in\G$, including degenerate and unipotent elements, where $\G$ is one of $\O_{p,q}(\mathbb{F})$, $\Sp_{2n}(\mathbb{F})$ and $\U(p,q)$. Our main results (Theorems \ref{stabw2}--\ref{stabwU}) give a constructive characterization of these centralizers via systems of linear matrix equations. 
In some cases, this yields a simple explicit description of the centralizer, while in the remaining cases it leads to a finite recursive algorithm that gives a parametrization of the resulting group. 
After an explicit conjugation, centralizers 
are realized as subgroups of a group of rectangular block Toeplitz matrices, where certain entries are identified as free parameters, while any remaining entries are uniquely determined by the preceding ones. This also yields the dimension and the semidirect-product structure of the resulting centralizers.
Note that Theorem \ref{stabw2} provides a precise description of the unipotent radical $\R_u(\C_{\G}(U))$ for a unipotent element $U\in\G$.

We also study the corresponding isotropy groups for several structured matrix spaces.
Define $\Skew_{p,q} (\mathbb{F})\subset \mathbb{F}^{n \times n}$ and $\Ham_{2n}(\mathbb{F})\subset \mathbb{F}^{2n \times 2n}$ to be the 
$\mathbb{F}$-vector subspaces of $I_{p,q}$-\emph{skew-symmetric} and \emph{Hamiltonian} matrices, consisting of $A$ such that $A^{T}I_{p,q} = - I_{p,q} A$ and $A^{T}J_{n} = - J_{n} A$, respectively. Over $\mathbb{C}$, let $\Her (p,q)
$ be the $\mathbb{R}$-vector subspace of $I_{p,q}$-\emph{Hermitian} matrices, consisting of $A$ satisfying $A^{*}I_{p,q} = I_{p,q} A$. 

For $\G$ equal to one of $\O_{p,q}(\mathbb{F})$, $\Sp_{2n}(\mathbb{F})$, or $\U(p,q)$, let $\mathfrak{g}$ denote, respectively, $\Skew_{p,q}(\mathbb{F})$, $\Ham_{2n}(\mathbb{F})$ or $i\Her({p,q})$ ($I_{p,q}$-\emph{skew-Hermitian} matrices). We consider the
$\G$-similarity action on $\mathfrak{g}$:
\begin{align}\label{cb}  
\G \times \mathfrak{g} \to \mathfrak{g},\qquad (Q,A)\mapsto Q^{-1}AQ,
\end{align}
and denote the \emph{isotropy group}, or stabilizer, at $A\in \mathfrak{g}$, with respect to (\ref{cb}), by
\begin{align}\label{SGA}
\Sigma_{\G}(A):=\{Q\in \G \mid Q^{-1}A Q=A\}.
\end{align}
From the viewpoint of Lie theory, 
(\ref{cb}) is the adjoint  action
of the corresponding classical Lie group on its Lie algebra.
We also consider the action of $\U(p,q)$ on $\Her(p,q)$, which is equivalent, via multiplication by $i$, to the action on $i\Her(p,q)$.
Thus, the isotropy group of $A\in\Her(p,q)$ coincides with that of $iA\in i\Her(p,q)$.

Furthermore, for $\G$ and $\mathfrak{g}$ as in (\ref{cb}), we have the exponential map 
\[
\exp\colon \mathfrak{g} \to \G, \qquad  A\mapsto \sum_{j=0}^{\infty}\frac{1}{j!}A^j,
\]
which maps nilpotent elements of $\mathfrak{g}$ bijectively onto unipotent elements of $\G$.
Consequently, if $A\in \mathfrak{g} 
$ has a single eigenvalue $\lambda$, then 
%
\begin{equation}\label{eqAexp}
\Sigma_{\G}(A)=C_{\G}(\exp (A)).
\end{equation}
This correspondence establishes a link between isotropy groups and centralizers.

The isotropy groups studied here also arise naturally in
matrix congruence problems, including the geometry of skew-symmetric matrices \cite[Theorem 10]{Hua45}, \cite{Wan}, as well as the simultaneous reduction of triples of skew-symmetric matrices, with connections to systems of transpose-Sylvester equations; see, e.g., \cite{TeranDopi2}, \cite{DKS}.

For the main results and their proofs, it will be convenient to work in a more general $H$-matrix setting, with $H$ an arbitrary nonsingular matrix. The cases corresponding to $I_{p,q}$ and $J_n$ are then obtained by the appropriate transformations; see, e.g., Proposition \ref{conj}.
Our general approach (cf. Le\-mma \ref{genap}) is related to methods developed in \cite{TSOC} and \cite{TSOS}, 
where the iso\-tro\-py gro\-ups associated with ortho\-gonal *congruence on Hermitian and orthogonal similarity 
on symmetric matrices were studied. Here, however, several subtle adaptations are required (Lemma \ref{EqT}, Remark \ref{TToep}), while further technical difficulties arise, most notably in the real indefinite orthogonal and real symplectic cases (see Remark \ref{Rfw}). 
A further feature of the present context is that some of the resulting groups can be computed without recursive steps (see Remark \ref{OpI}(\ref{OpIi})).

%
%
%
%
%
\section{Preliminaries and canonical forms}\label{secIG}

\subsection{H-dependent matrix groups and structures}

We introduce the $H$-dependent setting and recall the transformation results and canonical forms needed for the proofs of the main theorems. This allows the canonical forms used below to be stated uniformly in a simple form.
For $H=H^T\hspace{-1mm}\in \GL_{n}(\mathbb{F})$ or $H=-H^T\hspace{-1mm}\in \GL_{2n}(\mathbb{F})$, 
the \emph{$H$-orthogonal} and \emph{$H$-symplectic} groups are
\[
\O_n(\mathbb {F},H)=\{R\in\GL_n(\mathbb {F})\mid R^T H R=H\}, \quad
\Sp_{2n}(\mathbb {F},H)=\{R\in\GL_{2n}(\mathbb {F})\mid R^T H R=H\},
\]
respectively, and
for $H=H^{*}\in \GL_{n}(\mathbb{C})$, the \emph{$H$-unitary group} is defined by
\[
\U(n,H)=\{R\in\GL_n(\mathbb {C})\mid R^* H R=H\}.
\]
Additionally, for $H=H^T\in GL_n(\mathbb{F})$ or $H=-H^T\in GL_{2n}(\mathbb{F})$, respectively, denote by $\Skew_n (\mathbb{F},H)$ and $\Ham_{2n}(\mathbb{F},H)$ the 
$\mathbb{F}$-vector spaces of $H$-\emph{skew-symmetric} and $H$-\emph{Hamiltonian} matrices, consisting of all $A$ satisfying $A^{T}H = - H A$. For $H^{*}=H\in \GL_n(\mathbb{C})$,
let $\Her_{}(n,H)$ denote the $\mathbb{R}$-vector space of $H$-\emph{Hermitian} matrices, consisting of all $A$ satisfying $A^{*}H = H A$. 
Thus $\Skew_{n} (\mathbb{F}):=\Skew_{n} (\mathbb{F}, I_n)$ and $\Her_{}(n):=\Her_{}(n,I_n)$ are 
the standard skew-symmetric and Hermitian matrices.

We next record two standard propositions used repeatedly in the sequel. The first reduces the centralizer and isotropy computations to the individual spectral components appearing in the canonical forms below.

\begin{proposition} \label{stabs11}
Let $\mathbb{F}$ be either $\mathbb{C}$ or $\mathbb{R}$, and let $\G \subset \GL_{n}(\mathbb{F})$ with $n=\sum_{j=1}^m n_j$.
For each $j\in \{1,\ldots,m\}$, let $H_j,\G_j\subset \GL_{n_j}(\mathbb{F})$, and set $H:=\bigoplus_{j=1}^m H_j$. Next, let 
\begin{equation*}
A=\bigoplus_{j=1}^{m}A_j^{}, \qquad
R=\bigoplus_{j=1}^{m}R_j^{},
\end{equation*}
where matrices $A_1,\ldots,A_m$ have pairwise disjoint spectra (i.e. $A_j$, $A_k$ have no common eigenvalues for $j\neq k$), and likewise, matrices $R_1,\ldots,R_m$ have pairwise disjoint spectra.
Moreover, suppose exactly one of the following conditions hold:
\begin{enumerate}
\item $\G=\O_n(\mathbb{F},H)$, 
$H_j^T=H_j$, $A_j\in \Skew_{n_j}(\mathbb{F},H_j)$, $R_j\in \G_j=\O_{n_j}(\mathbb{F},H_j)$, for all $j$.
\item $\G=\Sp_n(\mathbb{F},H)$ with $n$ even, $H_j^T=-H_j$, $A_j\in \Ham_{n_j}(\mathbb{F},H_j)$, with $n_j$ even, $R_j\in \G_j=\Sp_{n_j}(\mathbb{F},H_j)$, for all $j$.
\item $\G=\U(n,H)$, $\mathbb{F}=\mathbb{C}$, $H_j^{*}=H_j$, $A_j\in \Her_{n_j}(H_j)$, $R_j\in\G_j=\U (n_j,H_j)$, for all $j$.
\end{enumerate}
Then 
\begin{align*}
\Sigma_{\G}(A)=\prod_{j=1}^{m}\Sigma_{\G_j}(A_j), \qquad \C_{\G}(R)=\prod_{j=1}^{m}\C_{\G_j}(R_j).
\end{align*}
\end{proposition}

The following proposition records the invariance of the centralizers and isotropy groups under a congruence transformation of $H$. It will allow us to work with arbitrary nonsingular $H$ and subsequently specialize to the standard forms.

\begin{proposition}\label{conj}
Assume $\mathbb{F}$ is either $\mathbb{C}$ or $\mathbb{R}$. 
Let $H,P \in \GL_{n}(\mathbb{F})$,
$\star\in\{*,T\}$, and
set $\widetilde{H}:=P^{\star} H P$.
Let $\G, \widetilde{\G} \subset \GL_{n}(\mathbb{F})$ and let $\mathfrak{g}, \widetilde{\mathfrak{g}}$ be vector subspaces of $\mathbb{F}^{n \times n}$.
Moreover, suppose exactly one of the following conditions hold:
\begin{enumerate}
\item 
$H^T=H$, 
$\star=T$, 
$\G=\O_n(\mathbb{F},H)$, $\widetilde{\G}=\O_n(\mathbb{F},\widetilde{H})$, $\mathfrak{g}= \Skew_{n}(\mathbb{F},H)$, $\widetilde{\mathfrak{g}}=\Skew_n(\mathbb{F},\widetilde{H})$. 
\item 
$H^T=-H$, 
$\star=T$, 
$\G=\Sp_n(\mathbb{F},H)$, $\widetilde{\G}=\Sp_n(\mathbb{F},\widetilde{H})$, $\mathfrak{g} = \Ham_{n}(\mathbb{F},H)$, $\widetilde{\mathfrak{g}}=\Ham_n(\mathbb{F},\widetilde{H})$.
\item 
$H^{*}=H$, $\mathbb{F}=\mathbb{C}$, 
$\star=*$, 
$\G=\U(n,H)$, $\widetilde{\G}=\U (n,\widetilde{H})$, $\mathfrak{g} = \Her_{}(n,H)$, $\widetilde{\mathfrak{g}}=\Her(n,\widetilde{H})$.
\end{enumerate}
Then we have
\begin{align*}
&\widetilde{\G}=P^{-1}\G P, \qquad \quad \widetilde{\mathfrak{g}}=P^{-1} \mathfrak{g} P,\\ 
&\Sigma_{\widetilde{\G}}(P^{-1}AP) =  P^{-1}\big( \Sigma_{\G}(A) \big) P^{}, \quad \C_{\widetilde{\G}}(P^{-1}RP) =P^{-1} \big(\C_{\G}(R)\big)P^{}, \qquad A\in \mathfrak{g}, \,\,R\in \G.
\end{align*}
\end{proposition}

\begin{remark}\label{indef}
All nonsingular skew-symmetric matrices over $\mathbb{F}$ are congruent, and all nonsingular complex symmetric matrices are congruent; see
\cite[Corollaries 4.4.4, 4.4.19]{HornJohn},\cite[Theorem 2.5]{Thomp}.
Thus, by Proposition~\ref{conj}, centralizers of elements of an $H$-orthogonal/$H$-symplectic group can be transferred via congruence to the corresponding standard forms of $H$.
Similarly, by Sylvester's inertia theorem, centralizers of elements of $\U(n,H)$ are isomorphic to those of the corresponding elements of $\U(p,q)$, where $H$ is congruent to $I_{p,q}$.
\end{remark}

\subsection{Canonical forms}

Throughout the rest of the paper, $H$ is understood to be nonsingular and to satisfy the symmetry condition associated with the structure under consideration.
Moreover, we use a unified notation $d\in\{0,1\}$ to treat the $H$-orthogonal/$H$-skew-symmetric and $H$-symplectic/$H$-Hamiltonian cases simultaneously: $d=0$ corresponds to the $H$-orthogonal and $H$-skew-symmetric settings, while $d=1$ corresponds to the $H$-symplectic and $H$-Hamiltonian settings.
This notation is summarized in Table~\ref{tab}.
In what follows, the dependence on $H$ and $d$ is suppressed when no confusion can arise.

%
\begin{table}[ht!]
\centering
\caption{Conventions for the parameter $d$ and the matrix $H$.}
\label{tab}
\begin{tabular}{c|c|c}
 $d=0$ & $d=1$ & unitary/Hermitian\\ \hline
 $H^T=H$ & $H^T=-H$ & $H^{*}=H$ \\ \hline
 $H$-skew-symmetric & $H$-Hamiltonian & $H$-Hermitian\\[2pt]
 $H$-orthogonal  & $H$-symplectic   &  $H$-unitary
\end{tabular}
\end{table}

We now recall the canonical forms for $H$-skew-symmetric/$H$-orthogonal, $H$-Ha\-mil\-to\-nian/$H$-symplectic and $H$-Hermitian/$H$-unitary matrices, under $H$-ortho\-go\-nal/\-$H$-symplectic/$H$-unitary similarity;
see, e.g., 
\cite[Ch. XI]{Gant}, 
\cite[Ch. 2]{LanRod},
\cite{Mehl}, \cite{RThomp}, \cite{Well}. 
These canonical forms are arranged according to their spectral configurations, and the notation introduced here will be used throughout the statements and proofs of our main results.

We denote the Jordan block with the eigenvalue $\lambda\in \mathbb{C}$, the exchange matrix, and the exchange matrix
with alternating signs, all of size $m\times m$, respectively, by
\begin{align}\label{Jblock}
  J_m(\lambda):=\begin{bmatrix}
                                                      \lambda    &  1       & \;     & 0    \\
						      \;     & \lambda     & \ddots & \;    \\     
						      \;     & \;      & \ddots &  1     \\
                                                      0     & \;      & \;     & \lambda   
                                   \end{bmatrix},\,\,																	
E_m:=\begin{bmatrix}
0  &                &      & 1\\
 &              &   1     &  \\
   &  \iddots  &  &  \\
1 &           &     & 0 \\
\end{bmatrix},\,\,																													
\Gamma_m:=\begin{bmatrix}
0  &                &      & 1\\
 &              &   -1     &  \\
   &  \iddots  &  &  \\
(-1)^{m}  &           &     & 0 \\
\end{bmatrix}.
\end{align}

Let either $A$ be $H$-skew-symmetric 
or $H$-Hamiltonian 
The Jordan canonical form of $A$ can then only contain blocks of the form $J_m(0)$ with $m+d$ odd, and pairs of blocks of the form $J_m(\lambda)\oplus J_m(-\lambda)$,
in which either $\lambda\in \mathbb{C}\setminus \{0\}$ or $\lambda=0$ and $m+d$ is even.
Moreover, there exists a nonsingular matrix $S$ such that $K:=S^T H S$ and $\mathcal{H}:=S^{-1} A S$ 
are of the following form:
\small
\begin{align}
\label{HK}
\mathcal{H}_{} & =
\bigoplus_{j}^{} 
\big(J_{\alpha_j}(\lambda_j)\oplus -(J_{\alpha_j}(\lambda_j))^{T}\big)
\oplus
\bigoplus_{k}^{} 
\big(J_{2\beta_k-d}(0)\oplus -(J_{2\beta_k-d}(0))^{T}\big)
\oplus
\bigoplus_l 
J_{2\gamma_l+d-1} (0),\\
%
K & =\bigoplus_{j}^{} 
\begin{bsmallmatrix}
0 & I_{\alpha_j}\\
(-1)^{d}I_{\alpha_j} & 0
\end{bsmallmatrix}
\oplus
\bigoplus_{k}^{} 
\begin{bsmallmatrix}
0 & I_{2\beta_k-d}\\
(-1)^{d} I_{2\beta_k-d} & 0
\end{bsmallmatrix}
\oplus
\bigoplus_l 
\varepsilon_l \Gamma_{2\gamma_l+d-1},\nonumber
\end{align}
\normalsize
where $\lambda_j\in \mathbb{C}^{}$ with $\Rea (\lambda_j)\geq 0$, and $\varepsilon_l=1$ (the complex case).
We next consider the $H$-orthogonal
and $H$-symplectic 
similarity canonical form of $R$.
Its Jordan canonical form can contain blocks $J_m(\pm 1)$ with $m+d$ odd, and pairs of blocks $J_{m}(\lambda)\oplus J_{m}(\lambda^{-1})$, in which either $\lambda\in \mathbb{C}\setminus \{0,\pm 1\}$ or $\lambda=\pm 1$ and $m+d$ is even.
Then there is a nonsingular $S$ such that $J:=S^T H S$ and $\mathcal{O}:=S^{-1} R S$ are of the form:
\small
\begin{align}
\label{NF2o}
\mathcal{O}_{}
= &
\bigoplus_j  
\big(\hspace{-1mm}\exp({J_{\alpha_j}(\varphi_j)})\oplus \exp({-(J_{\alpha_j}(\varphi_j))^T}) \big)
\oplus 
\bigoplus_k 
\big(\hspace{-1mm} \exp({J_{2\beta_k-d}(0)}) \oplus \exp({-(J_{2\beta_k-d}(0)})^T)\big)
\\
&\oplus\bigoplus_l \exp({J_{2\gamma_l+d-1}(0)})
 \oplus \bigoplus_m 
-\big( \exp({J_{2\delta_m-d}(0)}) \oplus \exp({-(J_{2\delta_m-d}(0)})^T)\big)\nonumber\\
&\oplus\bigoplus_n  
- \exp({J_{2\eta_n+d-1}(0)}),\nonumber
\\
%
%
J= &
\bigoplus_{j}^{} 
\begin{bsmallmatrix}
0 & I_{\alpha_j}\\
(-1)^{d}I_{\alpha_j} & 0
\end{bsmallmatrix}
\oplus
\bigoplus_{k}^{} 
\begin{bsmallmatrix}
0 & I_{\scriptscriptstyle 2\beta_k-d}\\
(-1)^{d} I_{\scriptscriptstyle 2\beta_k-d} & 0
\end{bsmallmatrix}
\oplus
\bigoplus_l 
\varepsilon_{\scriptscriptstyle l} \Gamma_{\scriptscriptstyle 2\gamma_l+d-1}\nonumber\\
&\oplus
\bigoplus_{m}^{} 
\begin{bsmallmatrix}
0 & I_{\scriptscriptstyle 2\delta_m-d}\\
(-1)^{d} I_{\scriptscriptstyle 2\delta_m-d} & 0
\end{bsmallmatrix}
\oplus
\bigoplus_n 
\vartheta_{\scriptscriptstyle n}\Gamma_{\scriptscriptstyle 2\eta_n+d-1},\nonumber
\end{align}
\normalsize
in which $\Ima (\varphi_j)\in [0,\pi]$, $e^{\varphi_j}\neq \pm 1$, and all $\varepsilon_l=\vartheta_{n}=1$ (the complex case).

In the complex case, $H$-skew-symmetric/$H$-orthogonal and $H$-Hamil\-to\-ni\-an/$H$-symplectic matrices are $H$-ortho\-go\-nally/$H$-symplectically similar if and only if they are similar (\cite[Corollary 22]{HornMerino}).

The real canonical forms require additional blocks.
The real Jordan canonical form of real $\H$-skew-symmetric 
or $\H$-Hamiltonian matrix
can only contain blocks of the form $J_m(0)$ with $m+d$ odd, 
pairs of blocks of the form $J_m(\lambda)\oplus J_m(-\lambda)$ with either $\lambda >0$ or $i\lambda \in \mathbb{R}^{}\setminus \{0\}$ or $\lambda=0$ and $m+d$ is even,
and quadruples of blocks of the form $J_n(\lambda)\oplus J_n(\overline{\lambda})\oplus J_n(-\lambda)\oplus J_n(-\overline{\lambda})$ with $\lambda \in \mathbb{C}^{}\setminus (\mathbb{R}\cup i\mathbb{R})$. 
Then there is a nonsingular matrix $S$ such that $\widetilde{K}:=S^T H S$ and $\widetilde{\mathcal{H}}:=S^{-1} A S$ 
are of the form (see, e.g., \cite[Theorems 2.7.1, 3.2.4]{LanRod}):
%
\begin{align*}
\widetilde{\mathcal{H}}_{}=
&
\mathcal{H}
\oplus
\bigoplus_{p}^{} 
\big(J_{\kappa_p}^{\mathbb{R}}(a_p,b_p)\oplus -(J_{\kappa_p}^{\mathbb{R}}(a_p,b_p))^{T}\big)
\oplus
\bigoplus_{q}^{} J_{\iota_q}(0,c_q),
\\
%
%
\widetilde{K}= &
K
\oplus
\bigoplus_p 
\begin{bsmallmatrix}
0 & I_{\kappa_p}\\
-I_{\kappa_p} & 0
\end{bsmallmatrix}
\oplus 
\bigoplus_{q}^{} 
\sigma_{q}\Delta_{\iota_q,d},
\end{align*}
in which $\mathcal{H}$, $K$ are of the form (\ref{HK}) with $\lambda_j>0$, $\varepsilon_l,\vartheta_n\in \{-1,1\}$ (the real case), and all $a_p,b_p,c_{q} >0$, $\sigma_q\in \{-1,1\}$, and where the real Jordan block with respect to a conjugate pair of eigenvalues $\lambda$ and $\overline{\lambda}$ is denoted by
\begin{align}\label{JblockR}
&  J_m^{\mathbb{R}}(a,b):=\begin{bsmallmatrix}
                                                      \Lambda(a,b)    &  I_2       & \;     & 0    \\
						      \;     & \Lambda(a,b)     & \ddots & \;    \\     
						      \;     & \;      & \ddots &  I_2     \\
                                                      0     & \;      & \;     & \Lambda(a,b)   
                                   \end{bsmallmatrix},\,\,\,																	
																	a,b \in \mathbb{R}\quad (2m\textrm{-by-}2m), \quad 
																	\Lambda(a,b):=\begin{bsmallmatrix}
a & b \\
-b  & a
\end{bsmallmatrix},
\end{align}
and the $m$-by-$m$ block matrix $\Delta_{m,d}$ is defined by
\begin{align} \label{Dmd}
&
\Delta_{m,d}:=\begin{bmatrix}
0  &                &      & J_1^{m+d-1}\\
 &              &   -J_1^{m+d-1}     &  \\
   &  \iddots  &  &  \\
(-1)^{m-1}J_1^{m+d-1}  &           &     & 0 \\
\end{bmatrix}, \,
\begin{array}{l}
J_1:=\begin{bsmallmatrix} 0 & 1\\ -1 & 0\end{bsmallmatrix} \\
\\
J_1^{m}=\left\{
\begin{array}{ll}
\hspace{-2mm}(-1)^{m/2}I_2, & \hspace{-1mm}m\textrm{ even}\\
\hspace{-2mm}(-1)^{(m-1)/2}J_1, & \hspace{-1mm}m \textrm{ odd}\
\end{array}
\right.\hspace{-2mm}.
\end{array}
\end{align}
Similarly, the real $\widetilde{J}$-orthogonal/$\widetilde{J}$-symplectic similarity
canonical form is:
\begin{align*}
\widetilde{\mathcal{O}}_{}=
&
\mathcal{O}
\oplus
\bigoplus_{p}^{} \Big(
\exp \big(J_{\kappa_p}^{\mathbb{R}}(a_p,b_p)\big)\oplus \exp \big(-(J_{\kappa_p}^{\mathbb{R}}(a_p,b_p))^{T}\big)\Big)
\oplus
\bigoplus_{q}^{} 
\exp \big(J_{\iota_q}(0,c_q) \big),\\
%
%
\widetilde{J}=
&
J
\oplus
\bigoplus_p 
\begin{bsmallmatrix}
0 & I_{\kappa_p}\\
-I_{\kappa_p} & 0
\end{bsmallmatrix}
\oplus 
\bigoplus_{q}^{} 
\sigma_{q}\Delta_{\iota_q,d},
\end{align*}
%
where $\mathcal{O}$, $J$ are of the form (\ref{NF2o}) for $\varphi_j>0$, $\varepsilon_l,\vartheta_n\in \{-1,1\}$ (the real case), and $a_p,b_p,c_{q} >0$, $\sigma_q\in \{-1,1\}$.

Finally, let $A$ be $H$-Hermitian
Its Jordan canonical form can then only contain blocks of the form $J_m(\lambda)$ with $\lambda$ real, and pairs of blocks of the form $J_m(\lambda)\oplus J_m(\overline{\lambda})$,
in which either $\lambda\in \mathbb{C}^{}\setminus \mathbb{R}$ (nonreal).
Moreover, there exists a nonsingular matrix $S$ such that $K:=S^{*} H S$ and $\mathcal{K}:=S^{-1} A S$ 
are of the following form:
%
\begin{align*}
\mathcal{K}_{}  & =
\bigoplus_{j}^{} 
\big(J_{\alpha_j}(\lambda_j)\oplus (J_{\alpha_j}(\overline{\lambda}_j))^{}\big)
\oplus
\bigoplus_{k}^{} 
J_{\beta_k}(\mu_k),\qquad
K=
\bigoplus_{j}^{} 
E_{2\alpha_j}
\oplus
\bigoplus_{k}^{} 
\varepsilon_k  E_{\beta_k},
\end{align*}
%
where $\lambda_j\in \mathbb{C}^{}\setminus \mathbb{R}$, $\mu_k\in \mathbb{R}^{}$ and $\varepsilon_k\in \{-1,1\}$
Finally, we introduce the $H$-unitary similarity canonical form of a matrix
$R$.
Its Jordan canonical form can contain blocks $J_m(\lambda)$ with $|\lambda|=1$, and pairs of block $J_{m}(\lambda)\oplus J_{m}(\overline{\lambda}^{-1})$ with $|\lambda| > 1$.
Then there is a nonsingular $S$ such that $J:=S^{*} H S$ and $\mathcal{U}:=S^{-1} R S$ are of the form:
\begin{align*}
\mathcal{U}_{}
 & = 
\bigoplus_j 
\big(\hspace{-1mm} \exp( i{J_{\alpha_j}(\lambda_j)}) \oplus \exp( i{(J_{\alpha_j}(\overline{\lambda}_j)})^{})\big)
\oplus
\bigoplus_k  
 \exp(i{J_{\beta_k}(\mu_k)}),\quad \Rea (\lambda_k)\geq 0,  |\mu_j|=1,  \\
%
%
J
& =
\bigoplus_{j}^{} 
E_{2\alpha_j}
\oplus
\bigoplus_{k}^{} 
\varepsilon_k E_{\beta_k}, \qquad \varepsilon_k\in \{-1,1\}.
\end{align*}
\normalsize

For the purposes of our proofs, we chose summands $\pm e^{A_j}$ of $H$-orthogonal/$H$-symplectic canonical form as exponentials of $H$-skew-symmetric/$H$-Hamiltonian blocks $A_j$. Similarly, the corresponding $H$-unitary summands are $e^{iA_j}$, where $A_j$ is $H$-Hermitian. This allows the centralizer problem for a group element to be related directly to the corresponding isotropy problem for a structured matrix.

%
%
%
%
%
\section{The main results}\label{MainR}

The canonical forms in Section~\ref{secIG} reduce the centralizer and isotropy group problems to Sylvester equations for Jordan blocks. Their solutions have rectangular Toeplitz structure and, after a suitable reordering of their entries (Lemma \ref{lemaP}), give rise to the group $\mathbb T^{\alpha,\mu}$; see \cite[Lemma 2.2]{TSOS}. It will serve as the ambient group for the descriptions in our main results. Given $\alpha=(\alpha_1,\ldots,\alpha_N)$ with 
$\alpha_{1}>\ldots >\alpha_{N}$ 
and $\mu=(m_1,\ldots,m_N)$, 
let $\mathbb{T}^{\alpha,\mu}$ consist of nonsingular
block matrices whose $(r,s)$-blocks are \emph{rectangular upper triangular Toeplitz} blocks of size $\alpha_r$-by-$\alpha_s$:
\begin{align}\label{0T0}
\mathcal{X}=[\mathcal{X}_{rs}]_{r,s=1}^{N}\in \mathbb{T}^{\alpha,\mu},\quad
&\mathcal{X}_{rs}=
\left\{
\begin{array}{ll}
\hspace{-2mm}[0\quad \mathcal{T}_{rs}], & \hspace{-1mm} \alpha_r<\alpha_s\\
\hspace{-2mm}\begin{bmatrix}
\mathcal{T}_{rs}\\
0
\end{bmatrix}, & \hspace{-1mm} \alpha_r>\alpha_s\\
\hspace{-2mm}\mathcal{T}_{rs},& \hspace{-1mm} \alpha_r=\alpha_s
\end{array}\right. \hspace{-1mm}, \quad
\begin{array}{ll}
\mathcal{T}_{rs}=T(A_0^{rs},\ldots,A_{b_{rs}-1}^{rs}),\\
\\
b_{rs}:=\min\{\alpha_s,\alpha_r\},
\end{array}
\end{align}
in which 
$A_0^{rr}\in \GL_{m_r}(\mathbb{F})$ for $r\in \{1,\ldots,N\}$, while $A_j^{rs}\in \mathbb{F}^{m_r\times m_s}$ for all admissible $r,s,j$, and 
a block \emph{upper triangular Toeplitz} matrix is
%
\begin{align}\label{UTT}
T(A_0,\ldots,A_{n-1}):=\begin{bmatrix}
  A_{0} & A_{1}                       & \ldots &    A_{n-1}  \\
0       & \ddots   &  \ddots    & \vdots \\
 \vdots &   \ddots           & \ddots   &   A_1 \\
0              & \ldots  &  0      & A_0
\end{bmatrix}.
\end{align}
When $A_0$ is the identity matrix, (\ref{UTT}) is upper \emph{unitriangu\-lar} Toeplitz. 
Two constructions used 
in the statements and proofs of our main results are given below.

\begin{example} \label{exW} (cf. \cite[Example 3.1]{TSOS}) Let $^{\star}$ denote either transpose $^{T}$ or conjugate transpose $^{*}$.
Fix $d\in \{0,1\}$, $\varepsilon\in \{-1,1\}$, and suppose $\alpha_{1}>\ldots >\alpha_{N}$,
$m_1,\ldots,m_N\in \mathbb{N}$. For $r\in \{1,\ldots,N\}$ let $B_{r}$ be nonsingular with
$B_{r}^{\star}=\varepsilon^{\alpha_r+d-1}B_{r}$, and let 
$Z_{j}^{r}$ 
satisfy $(Z_j^{r})^{\star}=-\varepsilon^{\alpha_r-j+d-1}Z_{j}^{r}$; all matrices are of size $m_r\times m_r$.
Define a matrix of the form (\ref{0T0}):
\small
\begin{align}\label{asZ}
&\mathcal{W}=\bigoplus_{r=1}^{N}T(I_{m_r},W_1^{r},\ldots,W_{\alpha_r-1}^{r})\in \mathbb{T}^{\alpha,\mu},\qquad
W_{j}^{r}:=B_{r}^{-1}\big(Z_{j}^{r}-\frac{1}{2}\sum_{k=1}^{j-1}\varepsilon^k(W_k^{r})^{\star}B_{r}W_{j-k}^{r}
\big).
\end{align}
\end{example}
\normalsize

\begin{example} \label{exF}
(cf. \cite[Example 3.2]{TSOS}) Let $\star$ stand for either transpose $(\star=T)$ or conjugate transpose $(\star=*)$. Assume $d\in \{0,1\}$ and $1\leq p< t\leq N$ with $\gamma:=\alpha_p-\alpha_t>0$, and let $B_{r}\in \GL_{m_r}(\mathbb{F})$ satisfy $B_r^{\star}=\varepsilon^{\alpha_r+d-1}B_{r}$ for $r\in \{p,t\}$.
Here $N_{\alpha}^{k}(X)$ denotes an $\alpha\times \alpha$ matrix with $X$ on the $k$-th super-diagonal (the main diago\-nal for $k=0$) and zeros elsewhere, and 
$C(n)
$ denotes the $n$th Catalan number.
For $0\leq k \leq \alpha_t-1$and $F \in \mathbb{F}^{m_t\times m_p}$ set $\mathcal{V}_{p,t}(F,k)\in \mathbb{T}^{\alpha,\mu}$ (see (\ref{0T0})) whose blocks are given by:
%
\small
\begin{align}
\label{Hptk}
&
[\mathcal{V}_{p,t}(F,k)]_{rs}=\left\{
\begin{array}{ll}
\hspace{-2mm}\bigoplus_{j=1}^{\alpha_r}I_{m_r}, & \hspace{-1mm} r=s\\
\hspace{-2mm}0,                        & \hspace{-1mm}  r\neq s 
\end{array}
\right. \hspace{-2mm}, \quad \{r,s\}\not\subset\{p,t\},\qquad\,
[\mathcal{V}_{p,t}(F,k)]_{tp}=
N_{\alpha_t}^{k}(F ),\\ 
&
[\mathcal{V}_{p,t}(F,k)]_{rr}=T(I_{m_r},V_1^{r},\ldots,V_{\alpha_r-1}^{r}), \quad r\in \{p,t\},\qquad
[\mathcal{V}_{p,t}(F,k)]_{pt}=
N_{\alpha_t}^{k}\big(-B_{p}^{-1}F^{\star}B_{t}\big),\nonumber
\end{align}
%
\small
\[
\begin{array}{l l}
\hspace{-3mm}\begin{array}{l}
V_{j}^{p}:=\left\{\begin{array}{ll}
\hspace{-2mm}a_{n-1}B_{p}^{-1}(F^{\star}B_{t}FB_{p}^{-1})^{n}B_p^{}, & \hspace{-2mm} j=n(2k+
\gamma)\\
\hspace{-2mm}0,                      & \hspace{-2mm} \textrm{otherwise}
\end{array}
\right.\hspace{-1mm}
\\
V_{j}^{t}:=\left\{\begin{array}{ll}
\hspace{-2mm} a_{n-1}B_{t}^{-1}(B_{t}FB_{p}^{-1}F^{\star})^{n}B_{t}, & \hspace{-2mm} j=n(2k+
\gamma)\\
\hspace{-2mm} 0,                      & \hspace{-2mm} \textrm{otherwise}
\end{array}
\right.\hspace{-1mm}
%
\end{array}
&                                               
\hspace{-6.0mm}
a_{n}:=
\left\{
\begin{array}{ll}
\hspace{-2mm} -\frac{1}{2^{2n+1}}C(n), & \hspace{-6.5mm} 
\varepsilon=1 \textrm{ or }\varepsilon=-1,\gamma \textrm{ even}\\
\hspace{-2mm} \frac{1}{2^{2n+1}}(-1)^{\frac{n-1}{2}}C(\frac{n-1}{2}), & \hspace{-2.5mm}
\varepsilon=-1,\gamma,n \textrm{ odd}\\
\hspace{-2mm} 0, &  
\hspace{-2.5mm}\textrm{otherwise},
\end{array}
\right.
\end{array}
\]
\normalsize
For instance, if $N=2$, $\alpha_1=4$, $\alpha_2=2$, $k=1$, $d=0$, 
$B_{1}=B_{2}=I$, $\star=T$, $\varepsilon=-1$, then:
\small
\begin{equation*}
%
\begin{bmatrix}[cccc|cc]
I_{} & 0 & 0 & -\tfrac{1}{2}F^{T}F  &  0  &  -F^{T} \\
0   & I_{} & 0   & 0  &  0  &   0\\
0   & 0   & I_{} & 0    &  0 & 0 \\
0   & 0   &  0  &  I_{}   & 0 &0 \\
\hline
0   & 0   & 0 & F                         &  I  & 0  \\
0   & 0   & 0   & 0              &  0    &   I  \\
\end{bmatrix}.
\end{equation*}
\normalsize
\end{example}
\normalsize

We now state our main results, which are proved in Sec.\ref{sec2}. 
For clarity, the results for the $H$-orthogonal/$H$-symplectic groups and the corresponding $H$-skew-symm\-etric/\-$H$-Hamiltonian matrices are divided according to the spectral configu\-ra\-ti\-ons in Sec. \ref{secIG}, while the $H$-Hermitian/$H$-unitary case is treated last.

The first theorem gives the centralizer/isotropy group in the unipotent/nilpotent case, together with its scalar $-1$ counterpart, and exhibits their semidirect-product structure within $\mathbb T^{\alpha,\mu}$. 

\begin{theorem}\label{stabw2} 
Assume $\alpha=(\alpha_1,\ldots,\alpha_N)$ with $\alpha_{1}>\ldots >\alpha_{N}$,
$m_1,\ldots,m_N\in \mathbb{N}$ and fix $d\in \{0,1\}$. Set $\mu=(\mu_{1},\ldots,\mu_{N})$ with
$\mu_{r}:=
\left\{
\begin{array}{ll}
\hspace{-2mm}m_r, &  \hspace{-1mm} d+\alpha_r \textrm{ odd}\\
\hspace{-2mm}2m_r, & \hspace{-1mm} d+\alpha_r \textrm{ even}
\end{array}
\right.\hspace{-2mm} $, $n:=\sum_{r=1}^N \alpha_r \mu_{r}$, 
and let $\mathbb{T}^{\alpha,\mu}$ be as in (\ref{0T0}), while $\Gamma_{\alpha_r}$ is as in (\ref{Jblock}).
Next, let $\theta,\vartheta_{r,j}\in \{-1,1\}$, 
with $\bigoplus_{j=1}^{m_r}\vartheta_{r,j}=I_{p_r,q_r}$
($m_r=p_r+q_r$).
Define
\begin{align}\label{defMH}
& \mathcal{A}^{}= 
\bigoplus_{r=1}^{N}\big(\bigoplus_{j=1}^{m_r}   A_{r}\big), \,\,\, A_{r}:=
\left\{
\begin{array}{ll}
\hspace{-1mm}J_{\alpha_r}(0)\oplus -(J_{\alpha_r}(0))^T,   &  \hspace{-1mm}  
d + \alpha_r \textrm{ even}\\
\hspace{-1mm}  J_{\alpha_r}(0),      &   \hspace{-1mm} 
d + \alpha_r \textrm{ odd}
\end{array}
\right.\hspace{-2mm}, \quad\,\,
\mathcal{R}=\theta e^{\mathcal{A}}, \\
 \nonumber
%
H= 
\bigoplus_{r=1}^{N} &
\big( \bigoplus_{j=1}^{m_r} 
H_{r,j}\big),
\,\,\,\, H_{r,j}:=
\left\{
\begin{array}{ll}
\hspace{-1.5mm}
\begin{bsmallmatrix}
0 & I_{\alpha_r}\\
(-1)^{d} I_{\alpha_r} & 0
\end{bsmallmatrix} ,   &  \hspace{-1.5mm} 
d + \alpha_r \textrm{ even}\\
\hspace{-1.5mm}
\vartheta_{r,j} \Gamma_{\alpha_r},      &   \hspace{-1.5mm} 
d + \alpha_r \textrm{ odd}
\end{array}
\right.\hspace{-2mm},\quad
\G= 
\left\{ 
\begin{array}{ll}
\hspace{-1.5mm}
\O_{n}(\mathbb{F},H),   &  \hspace{-1.5mm} 
d =0\\
\hspace{-1.5mm}
\Sp_{n}(\mathbb{F},H),      &   \hspace{-1.5mm} 
d = 1
\end{array}
\right.\hspace{-2mm}.
\end{align}

Then the centralizer of $\mathcal{R}$ and the isotropy group of $\mathcal{A}$ under $\G$-similarity coincide.
Moreover, there exists a nonsingular matrix $\Psi$ such that:
%
\begin{align}\label{CSgroup}
\Sigma_{\G}({\mathcal{A}^{}_{}}) = \Psi^{-1}\mathbb{X} \Psi, \qquad \mathbb{X}:=\mathbb{O} \ltimes \mathbb{V}\subset \mathbb{T}^{\alpha,\mu}.
\end{align}
\begin{enumerate}
\item
The subgroup $\mathbb{O}$ is the Levi factor of $\mathbb{X}$ and is given by:
%
\begin{align*}
\mathbb{O}=\prod_{r=1}^{N} \big(\prod_{j=1}^{\alpha_r} \G_{r}\big), \qquad\,\, 
G_{r}:= \left\{
\begin{array}{ll}
\hspace{-1mm} 
\O_{m_r}(\mathbb{F}), & \hspace{-1mm}  d+\alpha_r \textrm{ odd}\\
\hspace{-1mm} 
\Sp_{2m_r}(\mathbb{F}), & \hspace{-1mm}  d+\alpha_r \textrm{ even}
\end{array}
\right.\hspace{-2mm}.
\end{align*}
\item
The unipotent radical $\mathbb{V}\subset \mathbb{X}$
is generated by the matrices $\mathcal{W}$ in (\ref{asZ}) and by the matrices $\mathcal{V}_{p,t}(F,k)$ in
(\ref{Hptk}) for 
$\star=T$, $\varepsilon=-1$ and 
\small
$B_r := \left\{
\begin{array}{ll}
\hspace{-1mm} (-1)^{\alpha_r} I_{p_r,q_r} , & d+\alpha_r \textrm{ odd}\\
\hspace{-1mm} 
\begin{bsmallmatrix}
0 & I_{m_r}\\
-I_{m_r} & 0
\end{bsmallmatrix}, & d+\alpha_r \textrm{ even}
\end{array}
\right.\hspace{-2mm}$. 
\normalsize
\end{enumerate}

Moreover, $\mathbb X$ admits an explicit parametrization constructed algorithmically.
and it yields its $\mathbb{F}$-dimension:
\begin{align*}
\dim_{\mathbb{F}} (\mathbb{X})
=
\sum_{r=1}^{N}\big(\tfrac{1}{2} \alpha_r \mu_{r}^{2}+\sum_{s=1}^{r-1} \alpha_s \mu_{r} \mu_{s}\big) 
+(-1)^{d-1}\sum_{\alpha_r \textrm{ odd}}\tfrac{1}{2} \mu_{r}.
\end{align*}
\normalsize
%
\end{theorem}

For comparison, the classical description of the Levi factor of the centralizer of a unipotent $U\in \G$ is as follows. If 
$k$ is determined by the multiplicities of the Jordan blocks of $U$ of the corresponding size, then, up to conjugacy:
Next, for a unipotent $U\in \G$, the classical description of the Levi factor 
\[
\L (\C_{\G}(U))\cong
\left\{\begin{array}{ll}
\prod_{k \textrm{ odd}}^{}\O_{k}(\mathbb{F}) \times \prod_{k \textrm{ even}}^{}\Sp_{2k}(\mathbb{F}), & \G=\O_{n}(\mathbb{F})\\
\prod_{k \textrm{ odd}}^{}\Sp_{2k}(\mathbb{F})\times \prod_{k \textrm{ even}}^{}\O_{k}(\mathbb{F}), & \G=\Sp_{2n}(\mathbb{F})\\
\prod_{k} U(p_j,q_j), 
& \G=\U(p,q)
\end{array}
\right.\hspace{-2mm}.
\] 
Thus Theorem~\ref{stabw2} gives the full centralizer, including its unipotent radical.

\begin{example}
Suppose we are given an $H$-orthogonal matrix $\mathcal{R}$:
\[
\mathcal{R}=\bigoplus_{j=1}^{m}\Big(\exp \big( J_2(0)\big) \oplus \exp \big(-(J_2(0))^{T} \big) \Big)\oplus I_k, \qquad
H=\bigoplus_{j=1}^{m}\begin{bsmallmatrix}
0 & I_{2}\\
 I_{2} & 0
\end{bsmallmatrix}\oplus I_k,
\]
Using Algorithm \ref{algo} from Lemma \ref{EqT} and generators from Examples \ref{exW}, \ref{exF} for $d=0$, $\star=T$, $\varepsilon=-1$, the centralizer $\C_{\O_n(\mathbb{C},H)}(\mathcal{R})$, with $n:=4m+k$, is conjugate to a group consisting of matrices of the form:
\small
\[
\begin{bmatrix}[cc|c]
A_0 & A_0 J_m^{-1}Z & -A_0 J_m^{-1} F^T\widetilde{A}_0  \\ 
0   & A_0 & 0 \\
\hline
0   & F   & \widetilde{A}_0   
\end{bmatrix}
=
\big((A_0\oplus A_0)\oplus \widetilde{A}_0\big)
\begin{bsmallmatrix}
I_{2m} & J_m^{-1}Z & 0  \\ 
0   & I_{2m} & 0 \\
\hline
0   & 0   & I_{k}   
\end{bsmallmatrix}
\begin{bsmallmatrix}
I_{2m} & 0 & -J_m^{-1} F^T\widetilde{A}_0  \\ 
0   & I_{2m} & 0 \\
\hline
0   & \widetilde{A}_0^{T}F   & I_{2m}   
\end{bsmallmatrix},
\]
\normalsize
where
$\widetilde{A}_0\in \O_k(\mathbb{C})$, 
$A_0\in \Sp_{2m}(\mathbb{C},J_m)$,
$J_m=\begin{bsmallmatrix}
0 & I_{m}\\
- I_{m} & 0
\end{bsmallmatrix}$,
$F\in \mathbb{C}^{k\times 2m}$,
$Z=-Z^T\in \mathbb{C}^{2m\times 2m}$.

\end{example}

The distinction between $\lambda$ and $\widetilde\lambda$ in the next theorem separates the isotropy problem for $\mathcal A(\lambda)$ from the centralizer problem for its exponential $\mathcal R(\widetilde\lambda)$.

\begin{theorem}\label{stabw1}
For $\alpha=(\alpha_1,\ldots,\alpha_N)$ with $\alpha_{1}>\ldots >\alpha_{N}$, 
$\mu=(m_1,\ldots,m_N)$, set $n:=2\sum_{r=1}^N \alpha_r m_r$, and let $\mathbb{T}^{\alpha,\mu}$ be as in (\ref{0T0}) with $\mathbb{F}$ being $\mathbb{C}$ or $\mathbb{R}$.
Fix $d\in \{0,1\}$, $\lambda\in \mathbb{F}^{}\setminus \{0\}$, $\widetilde{\lambda} \in \mathbb{F}\setminus\{i k\pi\}_{k\in \mathbb{Z}}$, and define
\begin{align*}
\mathcal{A}^{}(\lambda)= &
\bigoplus_{r=1}^{N}\Big(\bigoplus_{j=1}^{m_r}\big(J_{\alpha_r}(\lambda)\oplus -(J_{\alpha_r}(\lambda))^T\big)\Big),\qquad \quad
\mathcal{R}(\widetilde{\lambda})=e^{\mathcal{A}(\widetilde{\lambda})},\\
H= &
\bigoplus_{r=1}^{N} \Big( \bigoplus_{j=1}^{m_r} 
\begin{bsmallmatrix}
0 & I_{\alpha_r}\\
 (-1)^{d} I_{\alpha_r} & 0
\end{bsmallmatrix}\Big), \qquad
\G= 
\left\{
\begin{array}{ll}
\hspace{-1.5mm}
\O_{n}(\mathbb{F},H),   &  \hspace{-1.5mm} 
d = 0\\
\hspace{-1.5mm}
\Sp_{n}(\mathbb{F},H),      &   \hspace{-1.5mm} 
d = 1
\end{array}
\right.\hspace{-2mm}.
\end{align*}
Then the isotropy group $\Sigma_{\G}\big({\mathcal{A}^{}(\lambda})\big)$ and the centralizer $\C_{\G}({\mathcal{R}^{}(\widetilde{\lambda})})$ 
coincide for $\lambda=\widetilde{\lambda}\in \mathbb{F}\setminus\{i k\pi\}_{k\in \mathbb{Z}}$, and for $\lambda\in \mathbb{F}^{}\setminus \{0\}$, $\widetilde{\lambda} \in \mathbb{F}\setminus\{i k\pi\}_{k\in \mathbb{Z}}$ they are conjugate to the group 
\[
\mathbb{X}:=\{\mathcal{X}\oplus (\mathcal{X}^{-1})^{T}\mid \mathcal{X}\in \mathbb{T}^{\alpha,\mu}\}\subset \mathbb{T}^{\alpha,\mu}\oplus \mathbb{T}^{\alpha,\mu};\,\quad
\dim_{\mathbb{F}}(\mathbb{X})=\sum_{r=1}^{N} m_r\big(\alpha_r m_r+2\sum_{s=1}^{r-1}\alpha_s m_s\big).
\]
\end{theorem}

The next two theorems treat the real blocks corresponding to the nonreal spectral configurations in the real $H$-orthogonal and $H$-symplectic groups; cf. the real canonical forms in Section~\ref{secIG}.

\begin{theorem}\label{stabw1Ri}
For $\alpha=(\alpha_1,\ldots,\alpha_N)$ with $\alpha_{1}>\ldots >\alpha_{N}$, 
$\mu=(m_1,\ldots,m_N)$, set $n:=2\sum_{r=1}^N \alpha_r m_r$, and let $\mathbb{T}^{\alpha,\mu}$ be as in (\ref{0T0}) with $\mathbb{F}=\mathbb{C}$.
Next, let $\vartheta_{r,j}\in \{-1,1\}$, 
with $\bigoplus_{j=1}^{m_r}\vartheta_{r,j}=I_{p_r,q_r}$
($m_r=p_r+q_r$).
Fix $d\in \{0,1\}$, $b>0$ and define
\begin{align*}
& 
\mathcal{A}^{}:=\bigoplus_{r=1}^{N}\Big(\bigoplus_{j=1}^{m_r} J_{\alpha_r}^{\mathbb{R}}(0,b)\Big), \qquad
\mathcal{R}:=e^{\mathcal{A}}, \\
& 
H=\bigoplus_{r=1}^{N} \Big( \bigoplus_{j=1}^{m_r}  
\vartheta_{r,j}
\Delta_{\alpha_r,d}
\Big), \qquad
\G= 
\left\{
\begin{array}{ll}
\hspace{-1.5mm}
\O_{n}(\mathbb{R},H),   &  \hspace{-1.5mm} 
d = 0 \\
\hspace{-1.5mm}
\Sp_{n}(\mathbb{R},H),      &   \hspace{-1.5mm} 
d = 1
\end{array}
\right.\hspace{-2mm}.
\end{align*}
Then 
$\Sigma_{\G}(\mathcal{A}^{})$ and 
$\C_{\G}({\mathcal{R}^{}})$ 
coincide and they are conjugate (over $\mathbb{C}$) to the group 
\[
\mathbb{X}\oplus (\mathbb{X}^{-1})^{T}, \qquad \quad \mathbb{X}:=\mathbb{O} \ltimes \mathbb{V} \subset \mathbb{T}^{\alpha,\mu},
\]
where $\mathbb{O}:=\prod_{r=1}^N \prod_{j=1}^{\alpha_r} U(m_r,B_r)$, and
$\mathbb{V}\subset \mathbb{X}$ is a unipotent normal subgroup, generated by the matrices 
$\mathcal{W}$ in (\ref{asZ}) and by the matrices $\mathcal{V}_{p,t}(F,k)$ in
for $\mathcal{F}=\mathcal{C}$, $\star=*$, $\varepsilon=-1$, 
$B_r=\left\{ 
\begin{array}{ll}
\hspace{-1mm}i^{\alpha_r+d} I_{p_r,q_r}, & \hspace{-1mm} \alpha_r+d \textrm{ odd} \\
\hspace{-1mm}i^{\alpha_r+d-1} I_{p_r,q_r}, & \hspace{-1mm} \alpha_r+d \textrm{ even}  \end{array} \right.
$\hspace{2mm}.

Moreover, $\mathbb X$ admits an explicit parametrization given by an algorithm, 
and its $\mathbb{R}$-dimension is
\begin{align*}
\dim_{\mathbb{R}} (\mathbb{X})=\sum_{r=1}^{N}\big(m_r^{2}\alpha_r+2\sum_{s=1}^{r-1}\alpha_s m_r m_s\big).
\end{align*}
\end{theorem}

\begin{theorem}\label{stabw2Ri}
For $\alpha=(\alpha_1,\ldots,\alpha_N)$ with $\alpha_{1}>\ldots >\alpha_{N}$, 
$\mu=(m_1,\ldots,m_N)$, set $n:=2\sum_{r=1}^N \alpha_r m_r$, and let $\mathbb{T}^{\alpha,\mu}$ be as in (\ref{0T0}) with $\mathbb{F}=\mathbb{C}$.
Fix $d\in \{0,1\}$, $a,b>0$ and define
\begin{align*}
& 
\mathcal{A}^{}:=\bigoplus_{r=1}^{N}\Big(\bigoplus_{j=1}^{m_r}\big(J_{\alpha_r}^{\mathbb{R}}(a,b)\oplus -(J_{\alpha_r}^{\mathbb{R}}(a,b))^T\big)\Big), \qquad
\mathcal{R}:=e^{\mathcal{A}}, \\
&
\G= 
\left\{
\begin{array}{ll}
\hspace{-1.5mm}
\O_{n}(\mathbb{R},H),   &  \hspace{-1.5mm} 
d = 0\\
\hspace{-1.5mm}
\Sp_{n}(\mathbb{R},H),      &   \hspace{-1.5mm} 
d = 1
\end{array}
\right.\hspace{-2mm}, \qquad 
H=\bigoplus_{r=1}^{N} \Big( \bigoplus_{j=1}^{m_r} 
\begin{bsmallmatrix}
0 & I_{\alpha_r}\\
 (-1)^{d} I_{\alpha_r} & 0
\end{bsmallmatrix}\Big).
\end{align*}
Then 
$\Sigma_{\G}\big(\mathcal{A}^{})\big)$ and 
$\C_{\G}(\mathcal{R}^{})$
coincide, and they are conjugate (over $\mathbb{C}$) to the group 
\[
\mathbb{X}:=\{\mathcal{X}\oplus (\mathcal{X}^{-1})^{T}\oplus \overline{\mathcal{X}}\oplus (\overline{\mathcal{X}}^{-1})^{T} \mid \mathcal{X}\in \mathbb{T}^{\alpha,\mu}\}, \quad
\dim_{\mathbb{R}}(\mathbb{X})=2\sum_{r=1}^{N} m_r\big(\alpha_r m_r+2\sum_{s=1}^{r-1}\alpha_s m_s\big).
\]
\end{theorem}

The final theorem treats the $H$-Hermitian/$H$-unitary case.

\begin{theorem}\label{stabwU}
For $\alpha=(\alpha_1,\ldots,\alpha_N)$ with $\alpha_{1}>\ldots >\alpha_{N}$, 
$\mu=(m_1,\ldots,m_N)$, set $n:=\sum_{r=1}^N \alpha_r m_r$, and let $\mathbb{T}^{\alpha,\mu}$ be as in (\ref{0T0}) with $\mathbb{F}=\mathbb{C}$.
Fix $\lambda\in \mathbb{C}$, $\vartheta_{r,j}\in \{-1,1\}$ with $\bigoplus_{j=1}^{m_r}\vartheta_{r,j}=I_{p_r,q_r}$,
and define further
\begin{align*}
& 
\mathcal{A}_{}
=
\bigoplus_{r=1}^{N} \bigoplus_{j=1}^{m_r} A_{r}, \quad
A_r
=\left\{
\begin{array}{ll}
\hspace{-1mm}J_{\alpha_r}(\lambda), & \hspace{-1mm}\lambda\in \mathbb{R}, \\
\hspace{-1mm}J_{\alpha_r}(\lambda)\oplus (J_{\alpha_r}(\overline{\lambda})), & \hspace{-1mm} \lambda\in \mathbb{C}\setminus \mathbb{R}
\end{array}
\right.\hspace{-2mm}, \qquad \quad \mathcal{U}
:=e^{i\mathcal{A}
}\\
& 
H
=\bigoplus_{r=1}^{N}\bigoplus_{j=1}^{m_r} H_{r,j}, \,\,
H_{r,j}(\lambda)=\left\{
\begin{array}{ll}
\hspace{-2mm} \vartheta_{r,j}  E_{\alpha_r}, & \hspace{-1mm} \lambda\in \mathbb{R}, \\
\hspace{-2mm} E_{2\alpha_r}, & \hspace{-2mm} \lambda\in \mathbb{C}\setminus \mathbb{R}
\end{array}
\right.\hspace{-2mm}, \quad 
\G
= \left\{\begin{array}{ll} \hspace{-2mm} \U(n,H
), & \hspace{-1mm} \lambda\in \mathbb{R}\\  \hspace{-2mm} \U(2n,H
), & \hspace{-2mm} \lambda\in \mathbb{C}\setminus \mathbb{R}   \end{array}   \right.\hspace{-2mm}.
\end{align*}
Then 
$\Sigma_{\G
}\big(\mathcal{A}^{}
\big)$ 
$\C_{\G
}(\mathcal{U}^{}
)$ 
coincide, and they are conjugate to the group 
\[
\mathbb{X}
:=\left\{\begin{array}{ll}
\mathbb{O} \ltimes \mathbb{V} \subset \mathbb{T}^{\alpha,\mu}  &   \lambda\in \mathbb{R}\\
\{\mathcal{X}\oplus (\mathcal{X}^{-1})^{*}\mid \mathcal{X}\in \mathbb{T}^{\alpha,\mu}\}\subset \mathbb{T}^{\alpha,\mu}\oplus \mathbb{T}^{\alpha,\mu};\,\quad & \lambda\in \mathbb{C}\setminus \mathbb{R}
\end{array}
\right.,
\]
where $\mathbb{O}:=\prod_{r=1}^{N} \prod_{j=1}^{\alpha_r} U(p_r,q_r)$, and $\mathbb{V}\subset \mathbb{X}$ is a unipotent normal subgroup, generated by matrices 
$\mathcal{W}$ in (\ref{asZ}) and by the matrices $\mathcal{V}_{p,t}(F,k)$ in
(\ref{Hptk}) for $\mathbb{F}=\mathbb{C}$, $\varepsilon=1$, $\star=*$, $B_r=I_{p_r,q_r}$. Moreover, $\mathbb X$ admits an explicit parametrization constructed algorithmically, and
\[
\dim_{\mathbb{R}}(\mathbb{X}
)=
\left\{ \begin{array}{ll}
\sum_{r=1}^{N}\big(\alpha_r m_{r}^{2}+2\sum_{s=1}^{r-1} \alpha_s m_{r} m_{s}\big), & \lambda\in \mathbb{R}\\
2\sum_{r=1}^{N} m_r\big(\alpha_r m_r+2\sum_{s=1}^{r-1}\alpha_s m_s\big), & \lambda\in \mathbb{C}\setminus \mathbb{R}
\end{array}
\right.\hspace{-1mm}.
\]
\end{theorem}

\begin{remark}\label{OpI}
\begin{enumerate}[label={(\roman*)},ref={\roman*},wide=0pt,itemsep=5pt]
\item \label{OpIi} In comparison with \cite[Theorem 2.3]{TSOC} and \cite[Theorem 3.2]{TSOS}, 
Theorems \ref{stabw1}-\ref{stabw2Ri} and Theorem \ref{stabwU} for real eigenvalues exhibit a particularly simple situation:
the group $\mathbb X$ has the particularly simple form $\mathbb {X}\cong\mathbb {T}^{\alpha,\mu}$. In the remaining cases, the parametrization of $\mathbb X$ is given recursively by Lemma~\ref{EqT}, through Algorithm~\ref{algo}.
Moreover, all proofs provide a matrix $\Psi$ satisfying $\C_{\G}({\mathcal{R}^{}_{}})=\Psi^{-1} \mathbb{X} \Psi $; see (\ref{psi0}), (\ref{QPsi}), (\ref{QPsiU}).
\item \label{OpIii} Theorems \ref{stabw1}-\ref{stabw2Ri} do not follow directly from Theorem \ref{stabw2}, which treats the nilpotent/unipotent case.
Indeed, the centralizers $\C_{\G}(S)$, $\C_{\G}(U)$ of the components in the Jordan-Chevalley decomposition $R=SU$ admit nice descriptions up to con\-ju\-ga\-tion by different matrices. Determining their intersection is therefore nontrivial.
\end{enumerate}
\end{remark}

As a consistency check, first consider the semisimple $S\in \G$, where the matrices are diagonalizable over $\mathbb{C}$. In our notation, this corresponds to $N=\alpha_1=1$.
\small
\[
\C_{\G}(S)=
\prod_{i}^{}\GL_{n_i}(\mathbb{C})
\times
\left\{
\begin{array}{ll}
\hspace{-1mm}\O_j(\mathbb{C})\times \O_k(\mathbb{C}),
& \hspace{-1mm} \G=\O_{n}(\mathbb{C})\\
\hspace{-1mm} \O_{j_1,j_2}(\mathbb{R})\times \O_{k_1,k_2}(\mathbb{R}) \times \prod_l \U(p_l,q_l) \times \prod_{r}^{}\GL_{m_r}(\mathbb{R}), 
& \hspace{-1mm} \G=\O_{p,q}(\mathbb{R})\\
\hspace{-1mm} \Sp_{2j}(\mathbb{C})\times \Sp_{2k}(\mathbb{C}), 
& \hspace{-1mm} \G=\Sp_{2n}(\mathbb{C})\\
\hspace{-1mm} \Sp_{2j}(\mathbb{R})\times \Sp_{2k}(\mathbb{R}) \times \prod_l \U(p_l,q_l) \times \prod_{r}^{}\GL_{m_r}(\mathbb{R}), 
& \hspace{-1mm} \G=\Sp_{2n}(\mathbb{R})\\
\hspace{-1mm} \prod_{j} U(p_j,q_j), 
& \hspace{-1mm} \G=\U(p,q)
\end{array}
\right.
\]
\normalsize
In particular, this applies to standard real skew-symmetric, real orthogonal, and unitary matrices, which are normal and hence diagonalizable over $\mathbb{C}$. 

\section{Solutions of a block Toeplitz matrix equation}\label{cereq}

Throughout this section, we use the notation $A^{\star}$ to denote either the transpose $A^T$ or the conjugate transpose $A^{*}$. Thus, statements involving $A^{\star}$ are understood to hold for both choices.

For $d\in \{0,1\}$, $\varepsilon\in \{-1,1\}$, 
$\alpha=(\alpha_1,\ldots,\alpha_N)$ with $\alpha_{1}>\ldots >\alpha_{N}$, 
$\mu=(m_1,\ldots,m_N)$, let 
\begin{align}\label{BBF}
&
\mathcal{F}=\bigoplus_{r=1}^{N}E_{\alpha_r}(I_{m_r}),\qquad
\mathcal{B}=\bigoplus_{r=1}^{N}\bigoplus_{j=1}^{\alpha_r} \varepsilon^{j-1} B_{r},
\end{align}
\normalsize
in which $B_{r}\in \GL_{m_r}(\mathbb{F})$ with $B_{r}^{\star}=\varepsilon^{\alpha_r+d-1}B_{r}$ and
$E_{\alpha}(I_{m}):=
\begin{bsmallmatrix}
 0                 &      & I_{m}\\
            &   \iddots     &  \\
I_{m}            &           &  0\\
\end{bsmallmatrix}
$
is the $\alpha\textrm{-by-} \alpha$ block exchange matrix with $I_m$ on the anti-diagonal.
To prove Theorem \ref{stabw2} it is essential 
to find all $\mathcal{X}\in \mathbb{T}^{\alpha,\mu}$ 
(see (\ref{0T0})) that solve the equation
\begin{equation}\label{eqFYFIY}
\mathcal{B}=\mathcal{F}\mathcal{X}^{\star}\mathcal{F}\mathcal{B} \mathcal{X}. 
\end{equation}
Writing $\mathcal X$ in rectangular block Toeplitz form and comparing block entries reduces (\ref{eqFYFIY}) to a finite system of linear matrix equations for the Toeplitz coefficients, whose solutions are described precisely in the next lemma.

%
%

The following lemma provides solutions to (\ref{eqFYFIY}). 
It is adapted from 
\cite[Lemmas 4.1, 4.2]{TSOS}, 
with several modifications to suit our setting; see Remark \ref{TToep}.

\begin{lemma}\label{EqT}
Given $\alpha=(\alpha_1,\ldots,\alpha_N)$ with $\alpha_{1}>\ldots >\alpha_{N}$, $\mu=(m_1,\ldots,m_N)$, $d\in\{0,1\}$, 
and let $\mathcal{B}$ be as in (\ref{BBF}). Then the set of solutions $\mathbb{X}$ of the form (\ref{0T0}) of the equation (\ref{eqFYFIY}) with $\star \in \{T,*\}$, $\varepsilon \in \{-1,1\}$, i.e. $\mathcal{X}=[\mathcal{X}_{rs}]_{r,s=1}^{N}\in \mathbb{X}\subset \mathbb{T}^{\alpha, \mu}$ with
\begin{equation*}\label{EqTX}
\mathcal{X}_{rs}=
\left\{
\begin{array}{ll}
\hspace{-1mm}[0\quad \mathcal{T}_{rs}], & \alpha_r<\alpha_s\\
\hspace{-1mm}\begin{bmatrix}
\mathcal{T}_{rs}\\
0
\end{bmatrix}, & \alpha_r>\alpha_s\\
\hspace{-1mm}\mathcal{T}_{rs},& \alpha_r=\alpha_s
\end{array}\right., \qquad 
\begin{array}{l}
\\
\mathcal{T}_{rs}=T\big(A_0^{rs},\ldots,A_{b_{rs}-1}^{rs}\big),\quad b_{rs}:=\min\{\alpha_s,\alpha_r\},\\
\\
A_j^{rs}\in \mathbb{F}^{m_r\times m_s},\quad A_0^{rr}\in GL_{m_r}(\mathbb{F}),
\end{array}
\end{equation*}
satisfies the following properties:
\begin{enumerate}
\item \label{EqTa}
The Algorithm~\ref{algo} gives a bijective parametrization of $\mathbb X$: the matrices $A_j^{rs}$ with $r>s$, the matrices $A_0^{rr}$ satisfying
$(A_0^{rr})^{\star} B_rA_0^{rr}=B_r$,
and the matrices $Z_j^r$ satisfying
$Z_j^r=-\varepsilon^{\alpha_r-j+d-1}(Z_j^r)^{\star}$
are freely chosen.
The algorithm then uniquely determines the diagonal entries $A_j^{rr}$ for $j\geq1$ and the remaining upper-diagonal entries $A_j^{rs}$ with $r<s$.
Conversely, every $\mathcal X\in\mathbb X$ is obtained in this way.

\begin{algorithm}
\caption{Calculating $A_j^{rs}$ for $r,s\in \{1,\ldots,N\}$, $j\in \{0,1,\ldots, b_{rs}-1\}$.}\label{algo}
\begin{algorithmic}
\vspace{1mm}
\State \hspace{-4mm} $\Psi^{krs}_{n}:= \sum_{i=0}^{n}\varepsilon^{i}(A_i^{kr})^{\star} B_k A_{n-i}^{ks}$
for $k,r,s\in \{1,\ldots,N\}$, $n\in \{0,1,\ldots b_{rs}-1\}$, $\star\in \{T,*\}$, $\varepsilon=\pm 1$.
\small
\State \hspace{-4mm} 
$\xi_n^{rs}:=
\sum_{i=1}^{n-1} \varepsilon^{i}(A_i^{rr})^{\star}B_{r}A_{n-i}^{rs}
+ (A_0^{rr})^{\star}B_{r}A_n^{rs}
+\left\{\begin{array}{ll}
\hspace{-2mm} 0, & \hspace{-3mm} s=r\\
\hspace{-2mm} \varepsilon^{n}(A_n^{rr})^{\star}B_{r}A_0^{rs}, &  \hspace{-3mm} s> r
\end{array}
\right.
$ for $n\in \{0,1,\ldots b_{rs}-1\}$.
\normalsize
\vspace{2mm}
\State \hspace{-4mm} \textbf{1:} $A_0^{rr}$ for $r\in\{1,\ldots,N\}$ is any solution of $B_{r}=(A_0^{rr})^{\star}B_{r}A_0^{rr}$,
\State \hspace{-3mm} \textbf{   } $A_j^{rs}$ for $r,s\in \{1,\ldots,N\}$ with $r>s$, $j\in \{0,\ldots,\alpha_{r}-1\}$ are taken arbitrarily.
\State \hspace{-4mm} \textbf{2:} $Z_j^{r}=-\varepsilon^{\alpha_r -j + d-1} (Z_j^{r})^{\star}\in \mathbb{F}^{m_r\times m_r}$ for $r\in \{1,\ldots,N\}$, $j \in \{1,\ldots, \alpha_r-1\}$ is chosen freely.
\For {$j=0:\alpha_1-1$}
    \If {
		$ j\in \{1,\ldots,\alpha_r-1\}$}
    \State $A_j^{rr}= A_0^{rr} B_{r}^{-1}\Big(Z_j^{r}-\frac{1}{2}\big(\xi_j^{rr}.
    +\sum_{k=1}^{r-1}\Psi_{j-\alpha_k+\alpha_r}^{krr}+\sum_{k=r+1}^{N}\Psi_{j-\alpha_{r}+\alpha_k}^{krr}\big)\Big)$, $r\in \{1,\ldots,N\}$
    \EndIf
    \For {$p=1:N-1$}
        \If {
				$r+p\leq N$, $j < \alpha_{r+p}$}
        \small
        \State \hspace{-3mm} $A_j^{r(r+p)}=-A_0^{rr} B_{r}^{-1}\Big(
        \xi_j^{r(r+p)}+\sum_{k=1}^{r-1}\Psi_{j-\alpha_k+\alpha_r}^{kr(r+p)}+ \sum_{k=r+1}^{r+p}\Psi_{j}^{kr(r+p)} 
				$\\
        \qquad \qquad \qquad \qquad \qquad \qquad \qquad \qquad \qquad \qquad \qquad \qquad $
				+\sum_{k=r+p+1}^{N}\Psi_{j-\alpha_{r+p}+\alpha_k}^{kr(r+p)}\Big)$   
\normalsize
        \EndIf
    \EndFor
\EndFor
\State \hspace{-4mm} (We define $\sum_{j=l}^{n}a_j=0$ for $l>n$, and the loop p =1 : N-1 is not performed for $N=1$.)
\end{algorithmic}
\end{algorithm}
%
%
%
The dimension of $\mathbb{X}$ is given by:
is:  
\small
\begin{align}
\label{cdim}
&\dim_{\mathbb{F}}(\mathbb{X})=\left\{\begin{array}{ll}
\hspace{-1mm}\sum_{r=1}^{N}\big(\tfrac{1}{2}m_r^{2}\alpha_r+\sum_{s=1}^{r-1}\alpha_s m_r m_s\big) +
(-1)^{d-1} \sum_{\alpha_r \textrm{ odd}}\tfrac{1}{2} m_r, & 
\star=T, \varepsilon=-1\\
\hspace{-1mm}\sum_{r=1}^{N}\big(\tfrac{1}{2}m_r^{}(m_r-1)\alpha_r+\sum_{s=1}^{r-1}\alpha_s m_r m_s\big), & 
\star=T, \varepsilon=1
\end{array} \right.\\
& \dim_{\mathbb{R}}(\mathbb{X})=\sum_{r=1}^{N}\big(m_r^{2}\alpha_r+2\sum_{s=1}^{r-1}\alpha_s m_r m_s\big)\qquad \textrm{for }\star=*, \mathbb{F}=\mathbb{C}. 
\nonumber
\end{align}
\normalsize
\item \label{EqTc}
The solution set $\mathbb{X}$ has a structure of a group, a semidirect product
\[
\mathbb{X}_{}=\mathbb{O}_{}\ltimes \mathbb{V}_{},\qquad \mathbb{O}_{}:=\prod_{r=1}^{N}\big(\prod_{j=1}^{\alpha_r} \G_r\big), \quad \G_r:=\left\{
\begin{array}{ll}
\hspace{-1mm}\O_{m_r}(\mathbb{F},B_r), & \hspace{-1mm} B_r^T=B_r\\
\hspace{-1mm}\Sp_{m_r}(\mathbb{F},B_r), & \hspace{-1mm} B_r^T=-B_r\\
\hspace{-1mm}\U_{}(m_r,B_r), & \hspace{-1mm} B_r^{*}=\pm B_r
\end{array}
\right.\hspace{-2mm},
\]
in which $\mathbb{V}_{}\subset \mathbb{T}_{}^{\alpha,\mu}$ is unipotent and each $\mathcal{V}\in \mathbb{V}_{}$ is of the form
$\mathcal{V}=\mathcal{W}\prod_{j=0}^{n}\mathcal{V}_{p_j,t_j}(F_j,k_j)$ (for some $n\leq \sum_{r=1}^{N}\sum_{s=1}^{r-1}\alpha_s$), where  
$\mathcal{W}$ is of the form (\ref{asZ}), and all 
$
\mathcal{V}_{p_j,t_j}(F_j,k_j)
$
are of the form (\ref{Hptk}).
\end{enumerate}

\end{lemma}

The following example illustrates the order in which the entries are determined by Algorithm~\ref{algo} in Lemma \ref{EqT} (\ref{EqTa}).

\begin{example}
We solve $\mathcal{F}\mathcal{X}^{T}\mathcal{F}\mathcal{B}\mathcal{X}=\mathcal{B}$ for $\mathcal{F}=E_3(I)\oplus E_2(I)$, $\mathcal{B}=(J\oplus -J\oplus J)\oplus (I \oplus -I)$, with the identity $I$, 
$J^T=-J$, and $\mathcal{X}$ partitioned conformally to $\mathcal{B}$: 
%
\small
\begin{align*}
\mathcal{B}=(\mathcal{F}\mathcal{X}^{T}\mathcal{F})(\mathcal{B}\mathcal{X})
=\begin{bmatrix}[ccc|c c]
A_1^{T} & B_1^{T} & C_1^{T}  &  P^{T}  &  Q^{T} \\  
0   & A_1^{T} & B_1^{T}     &  0      &   P^{T} \\ 
0   & 0   & A_1^{T}      &  0   & 0   \\  
\hline
0    & M^{T}  & N^{T}    &  A_2^{T}  &  B_2^{T} \\  
 0   & 0     & M^{T}     &  0       &   A_2 ^{T}  
\end{bmatrix}
\begin{bmatrix}[ccc|cc]
JA_1 & JB_1 & JC_1  &  JM  &  JN \\  
0   & -JA_1 & -JB_1    &  0  &   -J M  \\  
0   & 0   & J A_1     &  0  & 0   \\  
\hline
0     & P  & Q                &  A_2    &  B_2 \\  
0     & 0   & -P              &  0    &  -A_2    
\end{bmatrix}=
\end{align*}
\begin{align*}
=
\begin{bmatrix}[ccc|cc]
A_1^{T}JA_1 & A_1^{T}JB_1\hspace{-0.5mm}-\hspace{-0.5mm}B_1^{T} J A_1 & A_1^{T}J C_1\hspace{-0.5mm}+\hspace{-0.5mm}C_1^{T}J A_1\hspace{-0.5mm}-\hspace{-0.5mm}B_1^{T}J  B_1    & \hspace{-0.5mm} A_1^{T}JM\hspace{-0.5mm}+\hspace{-0.5mm}P^{T}A_2 \hspace{-0.5mm} &  A_1^{T}J N+B_1^{T}JM  \\ 
& & + P^{T}Q - Q^{T} P & & + P^{T}B_2- Q^{T}A_2 \\ 
  0         & -A_1^{T}JA_1 &  -A_1^{T} J B_1 +B_1^{T}J A_1    &  0   &   -A_1^{T}JM - P^{T}A_2 \\ 
0   &     0                  & A_1^{T}J A_1                        &  0  & 0             \\  
\hline 
   &    &                     &  A_{2}^{T} A_2       &   A_2^{T}B_2 - B_2^{T} A_2 \\ 
	 &    &                     &                       &     -M^{T}JM                          \\
   &    &                       &      0           &    + A_2^{T}A_2         
\end{bmatrix}
\end{align*}
\normalsize
\begin{itemize}[
wide=12pt,itemsep=2pt,leftmargin=25pt]
\item Choose $P,Q$ arbitra\-rily.
\item By comparing the main diagonal blocks we get $A_1$ and $A_2$ ($J=A_1^{T} J A_1$, $I=A_2^{T}A_2$).

\item The $(1,1)$-block of the upper right block gives $A_1^{T}JM+P^{T}A_2=0$, 
so $M=-A_1^{} J^{-1}P^{T}A_2$. 
\item The first superdiagonals yield $A_1^{T}JB_1-B_1^{T} J A_1=0$, $A_2^{T}B_2 - B_2^{T} A_2-M^{T}JM=0$ and $A_1^{T}JN-B_1^{T}M+P^{T}B_2-Q^{T}A_2$, thus $B_1=A_1J^{-1}Z_1$ for some $Z_1=-Z_1^{T}$,
$B_2=\frac{1}{2}A_2 M^{T}JM+ A_2 Z_2$ for some $Z_2=Z_2^{T}$,
$N=A_1 J^{-1}(B_1^{T}M-P^{T}B_2+Q^{T}A_2)$.
%
\item Finally, the second superdiagonal yields $A_1^{T}J C_1+C_1^{T}J A_1-B_1^{T}J B_1 + P^{T}Q - Q^{T} P=0$, 
therefore $C_1=A_1 J^{-1} Z_3+\frac{1}{2}A_1 J^{-1}(B_1^{T}J B_1 - P^{T}Q + Q^{T} P)
$ with $Z_3^T=Z_3$.
\end{itemize}
\end{example}

\begin{proof}[Proof of Lemma \ref{EqT}] 
(\ref{EqTa}) The proof of the Algorithm \ref{algo} has three steps. First, the symmetry relation (\ref{ETTE}) reduces the equation (\ref{eqFYFIY}) to its upper block-triangular part, and the Toeplitz structure reduces each block equation to its first row. Second, we show that the coefficients $A_j^{r(r+p)}$ with $p\geq 1$ occur linearly with coefficient \((A_0^{rr})^\star B_r\), whereas all other terms depend only on previously determined coefficients. This gives the recursion.

Since $\varepsilon^{\alpha_r-1}B_{r}^{\star}=\varepsilon^{d}B_{r}$, 
one easily verifies that 
\small
\begin{align*}
E_{\alpha_r}(I_{m_r})\big(\bigoplus_{j=1}^{\alpha_r}\varepsilon^{j-1}B_r\big)^{\star}  E_{\alpha_r}(I_{m_r})
=
\varepsilon^{\alpha_r-1} \bigoplus_{j=1}^{\alpha_r}\varepsilon^{j-1}  B_{r}^{\star}
=
\varepsilon^{d} \bigoplus_{j=1}^{\alpha_r}\varepsilon^{j-1}  B_{r},
\end{align*}
\normalsize
hence
\begin{align}
\label{ETTE}
(\mathcal{F}\mathcal{X}^{\star}\mathcal{F}\mathcal{B} \mathcal{X})^{\star}=\mathcal{X}^{\star}\mathcal{B}^{\star}\mathcal{F}\mathcal{X}\mathcal{F} =
\mathcal{X}^{\star}\mathcal{F}(\mathcal{F}\mathcal{B}^{\star}\mathcal{F})\mathcal{X}\mathcal{F}=
\varepsilon^{d}\mathcal{F}(\mathcal{F}\mathcal{X}^{\star}\mathcal{F}\mathcal{B}\mathcal{X})\mathcal{F},
\end{align}

Thus compare the first rows of the blocks in the upper triangular parts of $\mathcal{F}X^{\star}\mathcal{F}\mathcal{B} X$ and $\mathcal{B}$.  
By simplifying the notation with 
\[\mathcal{Y}:=\mathcal{B}\mathcal{X}, \qquad \widetilde{\mathcal{X}}:=\mathcal{F}X^{\star}\mathcal{F},
\] 
and using the standard $k$th row and column notation $\mathcal{M}_{(k)}$ and $\mathcal{M}^{(k)}$, the equation (\ref{eqFYFIY}) can be written as the following a system of equations:
\begin{align}
\label{f1}
(\mathcal{B}_{r(r+p)})_{1j}= & \big((\widetilde{\mathcal{X}}\mathcal{Y})_{r(r+p)}\big)_{1j} \qquad \qquad \qquad (1\leq j\leq \alpha_{r+p},\quad 0\leq p \leq N-r)  \nonumber
\\
= & (\widetilde{\mathcal{X}}_{rr})_{(1)}(\mathcal{Y}_{r(r+p)})^{(j)}+\sum_{k=r+1}^{N}(\widetilde{\mathcal{X}}_{rk})_{(1)}(\mathcal{Y}_{k(r+p)})^{(j)}
 +\sum_{k=1}^{r-1}(\widetilde{\mathcal{X}}_{rk})_{(1)}(\mathcal{Y}_{k(r+p)})^{(j)}.
\end{align}

First, to calculate $A_0^{rr}$ for $r\in \{1,\ldots,N\}$, we observe (\ref{f1}) for $p=0$, $j=1$. Since
\[
(\widetilde{\mathcal{X}}_{rk})_{(1)}=\left\{
\begin{array}{ll}
\begin{bsmallmatrix}
(A_0^{kr})^{\star} & * & \ldots & *
\end{bsmallmatrix}, & k\geq r\\
\begin{bsmallmatrix}
0 & * & \ldots & *
\end{bsmallmatrix}, & k<r
\end{array}
\right.,\qquad
(\mathcal{Y}_{kr})^{(1)}=
\left\{
\begin{array}{ll}
\begin{bsmallmatrix}
B_{k}A_0^{kr} \\
0\\
\vdots\\
0
\end{bsmallmatrix}, & k\leq r\\
0, & k>r
\end{array}
\right.,
\]
\vspace{-1mm}
we deduce $\sum_{k=1}^{N}(\widetilde{\mathcal{X}}_{rk})_{(1)}((\mathcal{Y})_{kr})^{(1)}=
(A_0^{rr})^{\star}B_{r}A_0^{rr}$, and it implies:
\begin{equation}\label{GABA}
B_{r}=(A_0^{rr})^{\star}B_{r}A_0^{rr}, \qquad r\in \{1,\ldots,N\}.
\end{equation}
To complete the Step 1 of Algorithm \ref{algo}, we arbitrarily set $\mathcal{X}_{rs}$ for $N\geq 2$, $r>s$.

Next, we inductively compute the remaining entries of $\mathcal{X}_{r(r+p)}$ for $0\leq p\leq N-r$. Throughout the following recursive computation, we assume
\[
0\leq j<b_{r(r+p)}=\alpha_{r+p}.
\]
Fix $p\in \{0,\ldots, N-1 \}$, $j\leq \alpha_r-1$, but not $p=j=0$, and solve (\ref{f1}) to get $A_j^{r(r+p)}$ (loops $j=0:\alpha_1-1$, $p=1:N-1$ 
of Algorithm \ref{algo}), assuming that we have already determined $A_{n}^{rs}$ for 
%
\begin{align}\label{induA1}
&j\geq 1, n\in \{0,\ldots,j-1\}, s\geq r 
\quad\textrm{or}\quad p\geq 1, n=j, r\leq s\leq r+p-1 \\
&\textrm{ or } \quad s\leq r, n\in \{0,\ldots,b_{rs}-1\}, N\geq 2, \qquad (1 \leq r, s\leq N).\nonumber
\end{align}
\normalsize

We have
\begin{align*}
\widetilde{\mathcal{X}}_{rk}
=
E_{\alpha_r}(I_{m_r})\mathcal{X}_{kr}^{\star} E_{\alpha_k}(I_{m_k})
&=
\left\{
\begin{array}{ll}
\hspace{-1mm}\begin{bsmallmatrix}
\widetilde{\mathcal{T}}_{rk}\\
0
\hspace{-1mm}\end{bsmallmatrix}, 
& \alpha_r >\alpha_k \\
\hspace{-1mm}\begin{bsmallmatrix}
0 & \widetilde{\mathcal{T}}_{rk}
\end{bsmallmatrix}, & \alpha_r<\alpha_k\\
\widetilde{\mathcal{T}}_{rk}, & \alpha_r=\alpha_k
\end{array}
\right.\hspace{-1mm}, \quad
\widetilde{T}_{rk}=
T\big((A_0^{kr})^{\star},\ldots,(A_{b_{kr}-1}^{kr})^{\star}\big),
\end{align*}
\vspace{-3mm}
%
\begin{align}\label{YSP}
&\mathcal{Y}_{ks} 
=
\left\{
\begin{array}{ll}
\hspace{-1mm}\begin{bsmallmatrix}
\mathcal{S}_{ks}\\
0
\end{bsmallmatrix}, 
& \alpha_k >\alpha_s \\
\hspace{-1mm}\begin{bsmallmatrix}
0 & \mathcal{S}_{ks}
\end{bsmallmatrix}, & \alpha_k<\alpha_s\\
\hspace{-1mm}\mathcal{S}_{ks}, & \alpha_k=\alpha_s
\end{array}
\right. \hspace{-2mm}, \quad  
\mathcal{S}_{ks} 
=T_{\varepsilon}\big(B_k A_0^{ks},\ldots, B_k A_{b_{ks}-1}^{ks}\big),
\end{align}
where $T_{\varepsilon}(A_0,\ldots,A_{\alpha-1}):=[T_{jk}]_{j,k=1}^{\alpha}$, $T_{jk}=0$ for $j>k$, $T_{(j+1)(k+1)}=-T_{jk}$, i.e.
\setlength{\arraycolsep}{4.4pt}
\small
\begin{align}\label{AltToep}
T_{\varepsilon} (A_0,\ldots,A_{\alpha-1}):=\begin{bmatrix}
  A_{0} & A_{1}         &  \ldots             & \ldots &    A_{\alpha-1}  \\
0       & \varepsilon A_0 & \varepsilon A_{1}   &  \ldots    & \varepsilon A_{\alpha-2} \\
 \vdots & \ddots            & \varepsilon^2 A_0              & \ddots &   \vdots \\ 
 \vdots &  & \ddots   & \ddots            &  \varepsilon^{\alpha-2}A_1\\
0       & \ldots            &  \ldots &  0      & \varepsilon^{\alpha-1}A_0
\end{bmatrix};
\end{align}
\normalsize
%
for $\varepsilon=-1$ this is an $\alpha$-by-$\alpha$ block \emph{alternating upper triangular Toeplitz} matrix.

To simplify the computations, we set ($k,r,s\in \{1,\ldots,N\}$, $n\in \{0,1,\ldots, b_{rs}-1\}$):
\small
\begin{align*}
\Psi^{krs}_{n}:=
\left\{
\begin{array}{ll}
\hspace{-1mm} 
\begin{bsmallmatrix}
(A_0^{kr})^{\star} & (A_1^{kr})^{\star} & \ldots & (A_{n}^{rr})^{\star} 
\end{bsmallmatrix}
\begin{bsmallmatrix}
B_k A_{n}^{ks} \\
\varepsilon B_k A_{n-1}^{ks} \\
\vdots \\
\varepsilon^{n}B_k A_{0}^{ks} \\
\end{bsmallmatrix}, & \hspace{-1mm} n\geq 0
\\
\hspace{-1mm} 0, & \hspace{-1mm} n<0
\end{array}
\right.
=
\left\{
\begin{array}{ll}
\hspace{-1mm}  \sum_{i=0}^{n}\varepsilon^{i}(A_i^{kr})^{\star} B_k A_{n-i}^{ks}, & \hspace{-1mm}n\geq 0\\
\hspace{-1mm} 0, & \hspace{-1mm} n<0
\end{array}
\right.\hspace{-2mm}
\end{align*}
\normalsize
and by using the fact $B_{r}^{\star}=\varepsilon^{\alpha_r+d-1}B_{r}$ for all $r$, we observe the following:
%
\begin{align}\label{simetrija}
(\Psi^{krs}_{n})^{\star}
&
= \sum_{i=0}^{n}\varepsilon^{i}(A_{n-i}^{ks})^{\star} B_k^{\star} A_{i}^{kr} 
= \varepsilon^{\alpha_k+d-1} \sum_{i=0}^{n}\varepsilon^{i}  (A_{n-i}^{ks})^{\star} B_k A_{i}^{kr} \\
& =\varepsilon^{\alpha_k+d-1} \varepsilon^{n}\sum_{i=0}^{n}\varepsilon^{n-i}  (A_{n-i}^{ks})^{\star} B_k A_{i}^{kr} 
=\varepsilon^{\alpha_k+d-1+n}\Psi^{ksr}_{n}.\nonumber
\end{align}
%
Furthermore, we obtain
\begin{align}\label{psijr}
(\widetilde{\mathcal{X}}_{rk})_{(1)}(\mathcal{Y}_{k(r+p)})^{(n+1)}
&=
\left\{
\begin{array}{ll}
\Psi_{n-\alpha_{r+p}+\alpha_k}^{kr(r+p)}, &  k\geq r+p+1\\
\Psi_{n}^{kr(r+p)},                       &  r+p \geq k\geq r+1, p\geq 1\\
\Psi_{n-\alpha_k+\alpha_r}^{kr(r+p)},     &  k\leq r\\
\end{array}
\right..
\end{align}
In particular, for $\xi_j^{r(r+p)}$ defined as in Algorithm \ref{algo}, we deduce:
%
\begin{align}
\label{xijrp}
(\widetilde{\mathcal{X}}_{rr})_{(1)}(\mathcal{Y}_{r(r+p)})^{(j+1)}
=\xi_j^{r(r+p)}+
\left\{\begin{array}{ll}
\hspace{-2mm}(A_0^{rr})^{\star}B_{r}A_j^{rr}+ \varepsilon^{j}(A_j^{rr})^{\star}B_{r}A_0^{rr}, &  \hspace{-1mm}
p=0\\
\hspace{-2mm}(A_0^{rr})^{\star}B_{r}A_j^{r(r+p)}, & \hspace{-1mm} p\geq 1
\end{array}
\right.\hspace{-1mm}. 
\end{align}
\normalsize

Next, we have
\small
\begin{align*}
\begin{array}{l}
(\widetilde{\mathcal{X}}_{rk})_{(1)}=
\left\{
\begin{array}{ll}
\hspace{-1mm}\begin{bsmallmatrix}
(A_0^{kr})^{\star} & (A_1^{kr})^{\star} & \ldots & (A_{\alpha_{r}-1}^{kr})^{\star}
\end{bsmallmatrix}, & \hspace{-1mm} k\geq r\\
\\
\hspace{-1mm}\begin{bsmallmatrix}
0& \ldots & 0 &(A_0^{kr})^{\star}  & \ldots &(A_{b_{kr}}^{kr})^{\star} 
\end{bsmallmatrix}, & \hspace{-1mm} k<r
\end{array}
\right.\hspace{-1mm},
\\
\\
(\mathcal{Y}_{(r+p)(r+p)})^{(j+1)}=
\begin{bsmallmatrix}
B_{r+p} A_j^{(r+p)(r+p)} \\
\varepsilon B_{r+p} A_{j-1}^{(r+p)(r+p)} \\
\vdots\\
\varepsilon^{j}B_{r+p} A_0^{(r+p)(r+p)}
\end{bsmallmatrix},
\end{array}
%
\hspace{-4mm}
(\mathcal{Y}_{k(r+p)})^{(j+1)}=
\left\{\begin{array}{ll}
\hspace{-1mm}\begin{bsmallmatrix}
B_k A_j^{k(r+p)} \\
\vdots\\
\varepsilon^{j}B_k A_0^{k(r+p)}\\
0\\
\vdots\\
0
\end{bsmallmatrix}, & \hspace{-4mm} r+p > k\\
\hspace{-1mm} \begin{bsmallmatrix}
B_k A_{j-\alpha_{r+p}-\alpha_k}^{k(r+p)} \\
\vdots\\
\varepsilon^{j-\alpha_{r+p}-\alpha_k} B_k A_0^{k(r+p)}\\
0\\
\vdots\\
0
\end{bsmallmatrix}, & \hspace{-2mm} k> r+p
\end{array}
\right.
\hspace{-2mm}.
\end{align*}
\normalsize
Hence, the last two terms in (\ref{f1}) for $j+1$ instead of $j$ ($N\geq r+1\geq 2$) are:
%
\begin{align}\label{thjrp}
\Xi (j,r,p):=
& \sum_{k=r+1}^{N}(\widetilde{\mathcal{X}}_{rk})_{(1)}(\mathcal{Y}_{k(r+p)})^{(j+1)}\\
=
&\left\{\begin{array}{ll}
\sum_{k=r+1}^{N}\Psi_{j-\alpha_{r}+\alpha_k}^{krr}, & j\geq 1, p=0\\
\sum_{k=r+1}^{r+p}\Psi_{j}^{kr(r+p)}
+\sum_{k=r+p+1}^{N}\Psi_{j-\alpha_{r+p}+\alpha_k}^{kr(r+p)}, & j\geq 0, p\geq 1. 
\end{array}
\right.,\nonumber\\
%
\label{lajrp}
\Lambda(j,r,p):=
&  \sum_{k=1}^{r-1}(\widetilde{\mathcal{X}}_{rk})_{(1)}(\mathcal{Y}_{k(r+p)})^{(j+1)}=\sum_{k=1}^{r-1}\Psi_{j-\alpha_k+\alpha_r}^{kr(r+p)}.
\end{align}
For $j,p\geq 0$ with $j+p\geq 1$ we define
\begin{equation}\label{Djrp}
D_j^{r(r+p)}:=\xi_j^{r(r+p)}+\Xi(j,r,p)+\Lambda(j,r,p).
\end{equation}
We combine (\ref{f1}) with (\ref{xijrp}), (\ref{thjrp}), (\ref{lajrp}), (\ref{Djrp}) to obtain:
\begin{align}
\label{eqATB1}
&(A_0^{rr})^{\star}B_{r}A_j^{r(r+p)}=-D_j^{r(r+p)},\qquad p\geq 1,\\
\label{eqATB2}
&(A_0^{rr})^{\star}B_{r}A_j^{rr}+\varepsilon^{j}(A_j^{rr})^{\star}B_{r}A_0^{rr}=-D_j^{rr}, \qquad j\geq 1\,\,\,\,\, (p=0).
\end{align}
Thus the equation for an off-diagonal entry $A_j^{r(r+p)}$ with $p\geq 1$ is linear, while the equation for a diagonal entry $A_j^{rr}$ with $j\geq 1$ has the form $A^\star X+\varepsilon^{\beta} X^\star A=B$.
From (\ref{simetrija}) we get 
$\Psi^{krr}_{j-\alpha_r+\alpha_k}=\varepsilon^{\alpha_r-j+d-1}(\Psi^{krr}_{j-\alpha_r+\alpha_k})^{\star}$, which further implies that $\xi_j^{rr}$, $\Xi(j,r,0)$, $\Lambda(j,r,0)$ and $D_j^{rr}$ are of the same structure, e.g., $(D^{rr}_{j})^{\star}=\varepsilon^{\alpha_r-j+d-1}D^{rr}_{j}$.

Since (\ref{GABA}) implies 
$A_0^{rr}B_{r}^{-1}=
((A_0^{rr})^{\star}B_{r})^{-1}$, 
the equation
(\ref{eqATB1}) has the solution $A_j^{r(r+p)}=-A_0^{rr} B_{r}^{-1}D_j^{r(r+p)}$ for $p\geq 1$.
Next, obtain $A_j^{rr}$ by solving (\ref{eqATB2}). 
Its solution is $X=(A^{\star})^{-1}(Z+\frac{1}{2}B)$ with 
$Z^{\star}=-\varepsilon^{\beta}Z$, where $A^{\star}=(A_0^{rr})^{\star}B_{r}$, $\beta=\alpha_r-j+d-1$, $B=-D_j^{rr}$, and hence
$(A^{\star})^{-1}=
A_0^{rr}(B_{r})^{-1}$. Thus $Z$ parametrizes precisely the homogeneous solutions of the equation for $A_j^{rr}$, $j\ge1$. Moreover, $D_j^{r(r+p)}$ with $p\geq 1$ depends only on $A_n^{rs}$, with $n,r,s$ satisfying (\ref{induA1}).

It remains to count the free parameters to obtain the dimension formula 
in (\ref{EqTa}). In case $\star=T$, $\varepsilon=-1$, each $A_{0}^{rr}$ contributes $\frac{1}{2}m_r (m_r-1)$, any $A_{j}^{rr}$ for $j\geq 1$ contributes $\frac{1}{2}m_r (m_r+(-1)^{\alpha_r-j+d})$, and any $A_{j}^{rs}$ for $r>s$ gives $m_r m_s$:
\small
\begin{align*}
\sum_{r=1}^{N}\sum_{s=1}^{r-1}\alpha_s m_r m_s
+& \hspace{-0.7mm}
\sum_{\alpha_r \textrm{ odd}}\hspace{-2.0mm}\big(
\tfrac{\alpha_r+\varepsilon^{d-1}}{4}(m_r^2-m_r)+
\tfrac{\alpha_r-\varepsilon^{d-1}}{4}(m_r^2+m_r)\big)\\
+& \hspace{-0.7mm}\sum_{\alpha_r \textrm{ even}}\hspace{-2.5mm}
\tfrac{\alpha_r}{4}\big(
(m_r^2-m_r)+(m_r^2+m_r)\big).
\end{align*}
\normalsize
Simplifying this expression yields the stated formula in (\ref{cdim}).
The cases $\star=T$, $\varepsilon=1$ and $\star=*$, $\mathbb{F}=\mathbb{C}$ are analogous.

\vspace{3mm}

\noindent
(\ref{EqTc}) Solutions of the equation (\ref{eqFYFIY}) form a group, since for any 
pair of solutions $\mathcal{X}_1,\mathcal{X}_2$ of (\ref{eqFYFIY}), their product $\mathcal{X}_1\mathcal{X}_2^{-1}$ is a solution as well:
%
\begin{align*}
\mathcal{F}(\mathcal{X}_1\mathcal{X}_2^{-1})^{\star}\mathcal{F}\mathcal{B} (\mathcal{X}_1\mathcal{X}_2^{-1})
&=\mathcal{F}(\mathcal{X}_2^{-1})^{\star}\mathcal{F}\big(\mathcal{F}\mathcal{X}_1^{\star}\mathcal{F}\mathcal{B} \mathcal{X}_1\big)\mathcal{X}_2^{-1}
=\mathcal{F}(\mathcal{X}_2^{-1})^{\star}\mathcal{F}\mathcal{B}\mathcal{X}_2^{-1}=\\
&=\mathcal{F}(\mathcal{X}_2^{-1})^{\star}\mathcal{F}\mathcal{B}(\mathcal{B}^{-1}\mathcal{F}\mathcal{X}_2^{\star}\mathcal{F}\mathcal{B})=\mathcal{B}.
\end{align*}

We now prove that among solutions of the equation (\ref{eqFYFIY}), we have matrices 
(\ref{Hptk}); these were observed in view of \cite[Examples 3.2]{TSOS}. 
To confirm this fact, 
we shall prove its consistency with Algorithm \ref{algo} in (\ref{EqTa}).

Consider $\mathcal{X}=[\mathcal{T}_{rs}]_{r,s=1}^N$ such that
\small
$\mathcal{T}_{rs}=\left\{
\begin{array}{ll}
\hspace{-1mm}\bigoplus_{j=1}^{\alpha_r}I_{m_r}, &  r=s\\
\hspace{-1mm}0,                        &  r\neq s 
\end{array}
\right.$ for 
$ \{r,s\}\not\subset\{p,t\}
$ and 
\normalsize
\small
\begin{align}
\label{Trs}
\mathcal{T}_{rr}=T(I_{m_r},V_1^{r},\ldots,V_{\alpha_r-1}^{r}), \,\,\, r\in \{p,t\}, \quad
\mathcal{T}_{tp}=
N_{\alpha_t}^{k}(F ),\quad
\mathcal{T}_{pt}=
N_{\alpha_t}^{k}\big(-B_{p}^{-1}F^{\star}B_{t}\big),
\end{align}
where $0\leq k \leq \alpha_t-1$. We inspect whether $\mathcal{X}$ of this form could be obtained by Algorithm \ref{algo}. The definition in the first line of the algorithm gives
\[
\Psi^{rpp}_{j-\alpha_p+\alpha_t}
=\left\{\begin{array}{ll}
F^{\star} B_t F, & r=t\\
0, &      k\not\in \{p,t\}
\end{array}\right., \qquad
\Psi^{rtt}_{j-\alpha_p+\alpha_t}
=\left\{\begin{array}{ll}
B_t F B_p^{-1} F B_t, & r=p\\
0, &      k\not\in \{p,t\}
\end{array} \right.,
\] 
and it yields the computation in the if statement ($j\in \{1, \ldots,\alpha_r-1\}$) of the first for loop of the algorithm for $r\in \{t,p\}$:
\small
\begin{align}\label{Veq}
&  V_j^r=B_r^{-1}\big(Z_j^{r}-\frac{1}{2}\sum_{n=1}^{j-1}\varepsilon^{n}(V_{n}^r)^{\star}B_r V_{j-n}^r\big)
-\frac{1}{2}
\left\{
\begin{array}{ll}
B_p^{-1}F^{\star} B_t F, & j=2k+\alpha_p-\alpha_t, \,\, r=p\\
F B_p^{-1} F B_t, & j=2k+\alpha_p-\alpha_t,\,\, r=t \\
0,             & 
j\neq 2k+\alpha_p-\alpha_t,\,\, r\in \{t,p\}
\end{array}
\right.,
\end{align}
\normalsize
where $(Z_j^{r})^{\star}=\varepsilon^{\alpha_r-j+d-1} Z_j^{r}$.
One can easily see that recursions (\ref{Veq}) for $Z_j^{r}=0$ give
\begin{align}
&
\qquad V_{j}^{r}:=\left\{\begin{array}{ll}
\hspace{-1mm}a_{n-1}B_{p}^{-1}(F^{\star}B_{t}FB_{p}^{-1})^{n}B_p^{}, & j=n(2k+\alpha_p-\alpha_t),\,\, r=p\\
\hspace{-1mm} a_{n-1}B_{t}^{-1}(B_{t}FB_{p}^{-1}F^{\star})^{n}B_{t}, & j=n(2k+\alpha_p-\alpha_t),\,\, r=t\\
\hspace{-1mm} 0,                      & \textrm{otherwise}
\end{array}
\right.\hspace{-1mm},\nonumber
\end{align}
where $a_n$ has the initial term $a_0:= -\frac{1}{2}$ and satisfies the functional equation: 
\begin{align*}
a_n  =  -\frac{1}{2} \sum_{j=1}^{n}\varepsilon^{j(2k+\alpha_p-\alpha_t)}a_{j-1} a_{n-j}
= -\frac{1}{2} \varepsilon^{\alpha_p-\alpha_t}\sum_{j=0}^{n-1}\varepsilon^{j (\alpha_p-\alpha_t)}a_{j} a_{n-1-j}, \quad n\geq 1.
\end{align*}
Let $f(t):=\sum_{n=0}^{\infty}a_n t^{n}$ be the generating function associated with $a_n$. For $\varepsilon=-1$, $\alpha_p-\alpha_t$ even or $\varepsilon=1$, we obtain the equation $f(t)=-\frac{1}{2}t(f(t))^2-\frac{1}{2}$, which yields $f(t)=-\frac{1}{t}\big(1-(1-t)^{\frac{1}{2}}\big)=-\sum_{n=1}^{\infty}\binom{1/2}{n}t^{n-1}$ and $a_{n}=(-1)^{n+1}\binom{1/2}{n+1}=-\frac{1}{2^{2n+1}}\frac{1}{n+1}{\binom{2n}{n}}$. For $\varepsilon=-1$, $\alpha_p-\alpha_t$ odd, we have $f(t)=-\frac{1}{2}tf(t)f(-t)-\frac{1}{2}$, hence $f(t)+f(-t)=-1$ and $f(t)=\frac{1}{2}tf(t)(1+f(t))-\frac{1}{2}$. Finally, $f(t)=-\frac{1}{2t}\big(2-t-(t^2+4)^{\frac{1}{2}}\big)=-\frac{1}{2}+\sum_{m=1}^{\infty}\frac{1}{4^n}\binom{1/2}{n}t^{2n-1}$ with $a_{2n}=0$, $a_{2n-1}=\frac{1}{4^{(n+1)/2}}\binom{1/2}{(n+1)/2}=\frac{(-1)^{n-1}}{2^{2n-1}n}{\binom{2n-2}{n-1}}$ for $n\geq 1$.
%
It is also straightforward to see the consistency of Algorithm \ref{algo} with $\mathcal{T}_{pt}$ in (\ref{Trs}).

Assume $\mathcal{X}\in \mathbb{T}^{\alpha,\mu}$ is a solution of (\ref{eqFYFIY}).
From (\ref{EqTa}), it follows that $((\mathcal{X})_{rr})_{11}=A_0^{rr}$ with $(A_0^{rr})^{\star}B_rA_0^{rr}=B_r$. Therefore, $\mathcal{X}=\mathcal{O}\mathcal{V}_0$, where $\mathcal{O}=\bigoplus_{r=1}^{N}\big(\bigoplus_{j=1}^{\alpha_r} A_0^{rr} \big)$, and $\mathcal{V}_0\in \mathbb{T}^{\alpha,\mu}$ has upper unitriangular diagonal blocks.
We now describe the inductive procedure of multiplying $\mathcal{V}_0$ by matrices 
of the form (\ref{Hptk}). Suppose that after $(j-1)$th steps of this process we obtain $\mathcal{Z}=\mathcal{V}_0\mathcal{V}_1\cdots \mathcal{V}_j\in \mathbb{T}^{\alpha,\mu}$, with $\mathcal{V}_1,\ldots, \mathcal{V}_j$ of the form (\ref{Hptk}), and such that the blocks under the main diagonal in the first $p-1$ columns vanish (i.e. $\mathcal{Z}_{rs}$ vanishes for $p,r>s$), and the first $\alpha_p-\alpha_{p+1}+n$ columns of $\mathcal{Z}_{rp}$ for $r>p$ vanish.
Let $t$ be the largest index such that $(\mathcal{Z}_{tp})_{1(\alpha_p-\alpha_{p+1}+n+1)}\neq 0$, i.e 
$
(\mathcal{Z}_{tp})_{(1)}=
\begin{bsmallmatrix}
0 & \ldots & 0 & R_{n-\alpha_{p+1}+\alpha_t}^{tp} & \ldots & R_{\alpha_{t}-1}^{tp}
\end{bsmallmatrix}$, with all $R_{j}^{tp}\in \mathbb{C}^{m_t\times m_p}$ and $R_{n-\alpha_{p+1}+\alpha_t}^{tp}\neq 0$. 
We multiply $\mathcal{Z}$ with $\mathcal{V}_{j}=\mathcal{V}_{p,t}(-R_{n-\alpha_{p+1}+\alpha_t}^{tp},n+\alpha_t-\alpha_{p+1})$ in (\ref{Hptk}). 
Then $\mathcal{Z}':=\mathcal{Z}\mathcal{V}_{j}$ is of the same form as $\mathcal{Z}$, but in addition $(\mathcal{Z}_{tp}')^{(n+\alpha_t-\alpha_{p+1}+1)}= 0$. Indeed, it is apparent that $\mathcal{Z}_{rs}'=0$ for $r, p >s$ and $(\mathcal{Z}_{rp}')_{1(\alpha_p-\alpha_{p+1}+n+1)}= 0$ for $r>t$, while a quick computation yields:
\begin{align*}
(\mathcal{Z}_{tp}')_{(1)}
=&
\begin{bmatrix}
0 & \ldots & 0 & R_{n-\alpha_{p+1}+\alpha_t}^{tp} & \ldots & R_{\alpha_{t}-1}^{tp}
\end{bmatrix}
T(I_{m_p},V_{1}^{pp},\ldots,V_{\alpha_p-1}^{pp})\\
&+
T(I_{m_t},V_{1}^{tt},\ldots,V_{\alpha_t-1}^{tt})
\begin{bmatrix}0 & \ldots & 0 & N_{\alpha_p-\alpha_t}^{k+\alpha_t-\alpha_{p+1}}(-R_{n-(\alpha_{p+1}-\alpha_t)}^{tp})  \end{bmatrix}\\
=& \begin{bmatrix}
0 & \ldots & 0 & S_{n-\alpha_{p+1}+\alpha_t+1}^{tp} & \ldots & S_{\alpha_{t}-1}^{tp}
\end{bmatrix}, \quad S_{j}^{sp}\in \mathbb{C}^{m_s\times m_p}.
\end{align*}
Eventually, this process yields a block upper unitriangular matrix $\mathcal{W}=\mathcal{V}_0\mathcal{V}_1\cdots \mathcal{V}_m$:
\[
\mathcal{X}=\mathcal{O}\mathcal{V}_0
=
\mathcal{O}\mathcal{W}\mathcal{V}_m^{-1}\cdots \mathcal{V}_1^{-1},
\qquad m\leq \sum_{r=1}^{N}\sum_{s=1}^{r-1}\alpha_s,
\]
where $\mathcal{V}_1,\ldots, \mathcal{V}_m$ are of the form (\ref{Hptk}).
Observe that the inverse $\mathcal{W}^{-1}$ is block upper unitriangular, and $\mathcal{V}_1^{-1},\ldots, \mathcal{V}_m^{-1}$ are of the form (\ref{Hptk}), as well. 
Applying Algorithm \ref{algo} for $\mathcal{X}_{rs}=\mathcal{W}_{rs}=0$ with $r>s$, yields that $\mathcal{W}$ is as in (\ref{asZ}). 
\end{proof}

\begin{remark}\label{TToep}
The present choice of $\mathcal B$ can likely be generalized to alternating upper triangular Toeplitz blocks. A corresponding generalization for $\star=T$, $\varepsilon=1$ is considered in \cite[Lemmas 4.1--4.2]{TSOS}.
%
\end{remark}

\section{Proofs of the main results}\label{sec2}

The general approach to proving our main results is described by the following elementary lemma.
\begin{lemma}\label{genap}
For a nonsingular matrix $H$, such that either $H^T=H$ or $H^T=-H$ or $H^{*}=H$, respectively, let $\mathcal{A}^{}$ be $H$-skew-symmetric or $H$-Ha\-mil\-to\-ni\-an or $H$-Hermitian normal form.
Denote by $\mathcal{J}(\mathcal{A})$ the Jordan canonical form of $\mathcal{A}$ and let $U$ be a corresponding transition matrix, i.e. $\mathcal{A}=U^{-1} \mathcal{J}(\mathcal{A}) U$. 
Then $Q$ is contained in the isotropy group of $\mathcal{A}$ with respect to $H$-orthogonal, $H$-symplectic, or $H$-unitary similarity, respectively, if and only if $X=UQU^{-1}$ is a solution of the system of equations:
\begin{align}\label{QK2}
\mathcal{J}^{}(\mathcal{A})X= X\mathcal{J}^{}(\mathcal{A}), \quad (U^{-1})^{\star} H U^{-1}=  & X^{\star} \big((U^{-1})^{\star} H U^{-1}\big) X, \,\,\, 
\star:=\left\{ \begin{array}{ll}\hspace{-2mm} T, & \hspace{-2mm} H^{}=\pm H^{T}\\
                             \hspace{-2mm} *, & \hspace{-2mm} H = H^{*}  \end{array}  \right. \hspace{-2mm}.
\end{align}

\end{lemma}

\begin{proof}
The isotropy group at $\mathcal{A}$ under $H$-orthogonal or $H$-symplectic or $H$-unitary similarity 
consists of all $H$-orthogonal or $H$-symplectic or $H$-unitary matrices $Q$ satisfying the Sylvester equation: 
\begin{equation}\label{HQQH}
\mathcal{A}^{}Q= Q\mathcal{A}^{},\quad Q^{\star} H Q=H  \qquad \qquad (Q^{-1}\mathcal{A}Q=\mathcal{A}).
\end{equation}
We now transform (\ref{HQQH}) to (\ref{QK2}) by inserting $\mathcal{A}=U^{-1} \mathcal{J}(\mathcal{A}) U$ and $Q=U^{-1}XU$.
\end{proof}

The core of the problem lies in linking the system of
equations (\ref{QK2}) to block Toeplitz structures.
This part of the proof appears to be substantially more involved than in \cite{TSOC} and \cite{TSOS} (see Remark \ref{Rfw}). 
In some cases, the reduction leads directly to $\mathbb T^{\alpha,\mu}$, while in the remaining cases Lemma~\ref{EqT} is applied. 

The isotropy groups results give the centralizers of the corresponding exponential matrices.
Indeed, when $\mathcal{J}(\mathcal{A})$ has one eigenvalue $\lambda$, then Proposition \ref{eJ0XXeJ0} shows that equations $\mathcal{J}(\mathcal{A})X=X \mathcal{J}(\mathcal{A})$ and 
$\exp(\mathcal{J}(\mathcal{A}))X=X \exp(\mathcal{J}(\mathcal{A}))$ have the same solutions. Furthermore, $\mathcal{A}=U^{-1}\mathcal{J}(\mathcal{A})U^{}$ yields $\exp (\mathcal{A})=U^{-1}\exp (\mathcal{J}(\mathcal{A}))U^{}$, thus solutions of (\ref{HQQH}) satisfy 
\[
(\varepsilon \exp(\mathcal{A}))Q= Q (\varepsilon \exp(\mathcal{A})),\qquad \varepsilon \in \{1,-1\}.
\]

Before proving Theorems \ref{stabw1} and \ref{stabw2}, we recall the classical result \cite[Ch. VIII]{Gant} on the 
(nonsingular) solutions of a homogeneous Sylvester equation, and observe that they are closely related to the group $\mathbb{T}^{\alpha,\mu}$.

\begin{proposition} \label{resAoXXA} 
Suppose $\mathcal{J}$ is the Jordan canonical form (with direct summands of the form (\ref{Jblock})) and let us consider the following matrix equation:
\begin{align}\label{eqJYYJ}
\mathcal{J}X=X\mathcal{J}
\end{align}  
\begin{enumerate}[label={\arabic*.},ref={\arabic*},
leftmargin=20pt]
\item \label{resAoXXA1}
Let further $\lambda_1,\ldots,\lambda_n\in \mathbb{C}$ be pairwise distinct eigenvalues of $\mathcal{J}=\bigoplus_{j=1}^{n}\mathcal{J}_j$,
in which $\mathcal{J}_j$ consists of all summands that correspond to the eigenvalue $\lambda_j$.
Then the solution of (\ref{eqJYYJ}) is of the form $X=\bigoplus_{j=1}^{n} X_j$, with $X_j$ a solution of $\mathcal{J}_j X_j=X_j\mathcal{J}_j$.

\item \label{resAoXXA2}
Assume $\mathcal{J}=\bigoplus_{r=1}^{N}\big(\bigoplus_{j=1}^{m_r} J_{\alpha_r}(\lambda)\big)$, $\lambda \in \mathbb{C}$ (i.e. $\mathcal{J}$ has precisely one eigenvalue $\lambda$) and $\alpha_1>\ldots>\alpha_N$. Then the solution of (\ref{eqJYYJ}) is $X=[X_{rs}]_{r,s=1}^{N}$, where every block $X_{rs}$ is further an $m_r$-by-$m$ block matrix with $\alpha_r$-by-$ \alpha_s$ blocks of the form
\begin{equation}\label{QTY}
\left\{\begin{array}{ll}
\begin{bmatrix} 
0 & T
\end{bmatrix}, & \alpha_r<\alpha_s \\
\begin{bmatrix}
T\\
0
\end{bmatrix}, & \alpha_r>\alpha_s\\
T, & \alpha_r=\alpha_s
\end{array}
\right.,
\end{equation}
with $T$ upper triangular Toeplitz matrix of size $b_{rs}\times b_{rs}$; $b_{rs}:=\min\{\alpha_r,\alpha_s\}$. 
\end{enumerate}
\end{proposition}

Proposition \ref{resAoXXA} (\ref{resAoXXA2}) remains valid also if we replace $J_{\alpha_r}(\lambda)$ by $\exp({J_{\alpha_r}(\lambda)})$.

\begin{lemma}\label{resAoXXAo} 
Assume $\mathcal{J}=\bigoplus_{r=1}^{N}\big(\bigoplus_{j=1}^{m_r} \exp({J_{\alpha_r}(\lambda)})\big)$, $\lambda \in \mathbb{C}$.
Then the solutions $X$ of the equation $\mathcal{J}X=X\mathcal{J}$ coincide with those described in Proposition \ref{resAoXXA} (\ref{resAoXXA2}). 
\end{lemma}

\begin{proof}
Since $\exp({J_{\alpha}(\lambda)})=e^{\lambda }\exp({J_{\alpha}(0)})
$ and $\exp({J_{\alpha}(0)})=S_{\alpha}J_{\alpha}(1)S_{\alpha}^{-1}$ for some transition matrix $S_{\alpha}$, the equation $\exp({J_{\alpha_r}(\lambda)})Y =Y \exp({J_{\alpha_s}(\lambda)})$ transforms to
\begin{align}
\label{eJ0XXeJ0}
&\exp({J_{\alpha_r}(0)})Y =Y \exp({J_{\alpha_s}(0)}),\\
\label{J1XXJ1}
&J_{\alpha_r}(1)\widetilde{Y} =\widetilde{Y} J_{\alpha_s}(1), \qquad \quad \widetilde{Y}=S^{-1}_{\alpha_r}YS_{\alpha_s}.
\end{align}
By Proposition \ref{resAoXXA} (\ref{resAoXXA2}) every solution $\widetilde{Y}$ of (\ref{J1XXJ1}) is of the form (\ref{QTY}), 
so the dimension of the solution space of (\ref{J1XXJ1}) (or (\ref{eJ0XXeJ0})) is $b_{rs}=\min\{\alpha_r,\alpha_s\}$. Since upper triangular Toeplitz matrices commute, any matrix of the form (\ref{QTY}) solves (\ref{eJ0XXeJ0}).
Thus these solutions form a $b_{rs}$-dimensional subspace, and hence constitute the whole solution space. 
\end{proof}

We also recall a technical lemma \cite[Lemma 3.3 (1)]{TSOC} (cf. \cite[Sec. 3.1]{Lin} and \cite[Sec. 2]{TSOS}), which enables us to transform solutions of a Sylvester equation (block matrices with upper triangular Toeplitz blocks) into block upper triangular Toeplitz matrices of the form (\ref{0T0}). It relies on the permutation matrices
\begin{equation}\label{perS}
\Omega_{\alpha,m}:=\left[e_1\;e_{\alpha+1}\;\ldots\;e_{(m-1)\alpha+1}\;e_2\;e_{\alpha+2}\;\ldots\;e_{(m-1)\alpha+2}\;\ldots\;e_{\alpha}\;e_{2\alpha}\;\ldots\;e_{\alpha m}\right],
\end{equation}
where $e_1,
\ldots,e_{\alpha m}$ is the standard orthonormal basis in $\mathbb{C}^{\alpha m}$. Post-multiplication by $\Omega_{\alpha,m}$ (pre-multiplication by $\Omega_{\alpha,m}^{T}$) puts the $k$-th, the $(\alpha+k)$-th, \ldots, the $((m-1)\alpha+k)$-th column (row) together for all $ k \in \{1,\ldots, \alpha\}$. For example:
\[
\Omega_{3,2}^{T}\begin{bmatrix}[cc|cc|cc]
a_1 & b_1 & a_2 & b_2 & a_3 & b_3 \\
0   & a_1 & 0   & a_2 & 0   & a_3\\
0   & 0   & 0   & 0   & 0   &  0 \\
\hline
a_4 & b_4 & a_5 & b_5 & a_6 & b_6 \\
0   & a_4 & 0   & a_5 & 0   & a_6\\
0   & 0   & 0   & 0   & 0   &  0 
\end{bmatrix}\Omega_{2,3}
=
\begin{bmatrix}[ccc|ccc]
a_1 & a_2 & a_3 & b_1 & b_2 & b_3 \\
a_4 & a_5 & a_6   & b_4 & b_5   & b_6\\
\hline
0   & 0   & 0   & a_1   & a_2   &  a_3 \\
0 &  0 &   0 & a_4 & a_5 & a_6 \\
\hline
0   & 0 & 0   & 0 & 0   & 0\\
0   & 0   & 0   & 0   & 0   &  0 
\end{bmatrix}.
\]

\begin{lemma}\label{lemaP}
Suppose
$X=[X_{rs}]_{r,s=1}^{N}$ is an $N$-by-$N$ block matrix whose block $X_{rs}=[(X_{rs})_{jk}]_{j,k=1}^{m_r,m_s}$ is an $m_r$-by-$m_s$ block matrix with blocks of size $\alpha_r\times \alpha_s$, and such that
\[
(X_{rs})_{jk}=
\left\{
\begin{array}{ll}
[0\quad T_{jk}^{rs}], & \alpha_{r}<\alpha_{s}\\
\begin{bmatrix}
T_{jk}^{rs}\\
0
\end{bmatrix}, & \alpha_{r}>\alpha_{s}\\
T_{jk}^{rs},& \alpha_{r}=\alpha_{s}
\end{array}\right., \qquad
\begin{array}{l}
T_{jk}^{rs}=T(a_{0,jk}^{rs},a_{1,jk}^{rs},\ldots,a_{b_{rs}-1,jk}^{rs})\in \mathbb{C}^{b_{rs}\times b_{rs}},\\
\\
b_{rs}:=\min \{\alpha_r,\alpha_s\}.
\end{array}
\]
Set $\Omega=\bigoplus_{r=1}^{N}\Omega_{\alpha_r,m_r}$, in which $\Omega_{\alpha_r,m_r}$ are defined by (\ref{perS}). Then 
\vspace{-1mm}
\begin{align}\label{0T02}
&\mathcal{X}:=\Omega^{T}X\Omega, \qquad \mathcal{X}=[\mathcal{X}_{rs}]_{r,s=1}^{N},\quad
\mathcal{X}_{rs}=
\left\{
\begin{array}{ll}
[0\quad \mathcal{T}_{rs}], & 
\alpha_r<\alpha_s\\
\begin{bmatrix}
\mathcal{T}_{rs}\\
0
\end{bmatrix}, & \alpha_r>\alpha_s\\
\mathcal{T}_{rs},& \alpha_r=\alpha_s
\end{array}\right.,
\end{align}
where $\mathcal{X}_{rs}$ is of size $\alpha_r \times \alpha_s$ and $\mathcal{T}_{rs}=
T(A_0^{rs},\ldots,A_{b_{rs}-1}^{rs})$ 
with $ A_{n}^{rs}:=[a_{n,jk}^{rs}]_{j,k=1}^{m_r,m_s}$.
\end{lemma}

\begin{proof}[Proof of Theorem \ref{stabw2}]
For $d=0$ or $d=1$, 
let $H$ and the $H$-skew-symmetric or $H$-Hamiltonian normal form $\mathcal{A}$, with eigenvalue $0$, be as in (\ref{defMH}):
\begin{align}
&\label{Hr}
H:=\bigoplus_{r=1}^{N} \big( \bigoplus_{j=1}^{m_r} 
H_{r,j}\big),
\qquad H_{r,j}=
\left\{
\begin{array}{ll}
\hspace{-1mm}
\begin{bsmallmatrix}
0 & I_{\alpha_r}\\
(-1)^{d} I_{\alpha_r} & 0
\end{bsmallmatrix} ,   &  \hspace{-1mm} 
d + \alpha_r \textrm{ even}\\
\hspace{-1mm}
\theta_{r,j} \Gamma_{\alpha_r},      &   \hspace{-1mm} 
d + \alpha_r \textrm{ odd}
\end{array}
\right.,\\
\label{Kcase2}
&\mathcal{A}^{}=
\bigoplus_{r=1}^{N}\big(\bigoplus_{j=1}^{m_r} A_{r}\big), \quad A_{r}:=
\left\{
\begin{array}{ll}
\hspace{-1mm}J_{\alpha_r}(0)\oplus -(J_{\alpha_r}(0))^T,   &  \hspace{-1mm}  
d + \alpha_r \textrm{ even}\\
\hspace{-1mm} J_{\alpha_r}(0),      &   \hspace{-1mm} 
d + \alpha_r \textrm{ odd}
\end{array}
\right.\hspace{-2mm},
\end{align}
in which $\theta_{r,j}\in{-1,1}$ with $\theta_{r,j}=1$ for $\mathbb{F}=\mathbb{C}$,
and $\Gamma_{\alpha_r}$ is defined by (\ref{Jblock}).
Next, the Jordan canonical form $\mathcal{J}(\mathcal{A})$ of $\mathcal{A}$ and its corresponding transition matrix $U$ (i.e. $\mathcal{A}=U^{-1}\mathcal{J}(\mathcal{A})U$) are:
\begin{align}\label{Jcase2}
\mathcal{J}(\mathcal{A})
& = \bigoplus_{r=1}^{N}\big( \bigoplus_{j=1}^{m_r} \mathcal{J}(\widetilde{A}_r) \big),
\quad \mathcal{J}(\widetilde{A}_r):=
\left\{
\begin{array}{ll}
\hspace{-1mm}J_{\alpha_r}(0) \oplus J_{\alpha_r}(0),   &  \hspace{-1mm} 
d + \alpha_r \textrm{ even}\\
\hspace{-1mm}J_{\alpha_r}(0),             &   \hspace{-1mm} 
d + \alpha_r \textrm{ odd}
\end{array}
\right.\hspace{-2mm},
\\
\label{Ur}
U
& =
\bigoplus_{r=1}^{N} \big( \bigoplus_{j=1}^{m_r} U_r \big), \qquad
U_r
:=
\left\{
\begin{array}{ll}
\hspace{-1.5mm} I_{\alpha_r}\oplus \Gamma_{\alpha_r}, & \hspace{-1.5mm} d + \alpha_r \textrm{ even}\\
\hspace{-1.5mm} I_{\alpha_r}, &  \hspace{-1.5mm} d + \alpha_r \textrm{ odd}
\end{array}
\right..
\end{align}

In view of (\ref{HQQH}), 
we have $\mathcal{J}(\mathcal{A})X=X\mathcal{J}(\mathcal{A})$ for $X=U Q U^{-1}$ with $U$ as in (\ref{Ur}) and $H$-orthogonal/$H$-symplectic $Q$.
Proposition \ref{resAoXXA} (\ref{resAoXXA1}) yields the solution of this equation $X=[X_{rs}]_{r,s=1}^{N}$ with an $\mu_r$-by-$\mu_s$ block matrix $X_{rs}$ whose blocks are $\alpha_r$-by-$\alpha_s$ rectangular Toeplitz
of the form (\ref{QTY}), where 
\small
$
\mu_r:=\left\{
\begin{array}{ll}
2m_r,  & d + \alpha_r \textrm{ even} \\
m_r, &  d + \alpha \textrm{ odd}
\end{array}
\right.
$.
\normalsize

Next, $H$-orthogonality or $H$-symplecticity of $Q=U^{-1} X U$ is described by the equation (\ref{QK2}):
\begin{align}\label{QTQcase2}
(U^T)^{-1} H U^{-1} =   &  X^T \big((U^T)^{-1} H U^{-1}\big) X.
\end{align}
We shall now write $(U^{-1})^{T} H U^{-1}$ as a product of block diagonal matrices, 
one with diagonal and one with anti-diagonal blocks; this is the trick of the proof.

We observe that $\Gamma_{\alpha_r}^{T}=\Gamma_{\alpha_r}^{-1}=(-1)^{\alpha_r +1}\Gamma_{\alpha_r}^{}$ and we deduce
\begin{align}
\label{UTHU}
& ((I_{\alpha_r} \oplus \Gamma_{\alpha_r} )^{-1})^{T} 
\begin{bmatrix}
0 & I_{\alpha_r}\\
(-1)^{d} I_{\alpha_r}& 0 
\end{bmatrix}
(I_{\alpha} \oplus \Gamma_{\alpha_r})^{-1} 
= 
\begin{bmatrix}
0 &  (-1)^{\alpha_r -1} \Gamma_{\alpha_r}\\
(-1)^{d} \Gamma_{\alpha_r} & 0 
\end{bmatrix}= \nonumber\\
= &
\begin{bmatrix}
0 & -E_{\alpha_r} F_{\alpha_r}\\
(-1)^{d+\alpha_r} E_{\alpha_r} F_{\alpha_r} & 0 
\end{bmatrix}
= 
(E_{\alpha_r} \oplus E_{\alpha_r})
\begin{bmatrix}
0 & -F_{\alpha_r}\\
(-1)^{d+\alpha_r} F_{\alpha_r} & 0 
\end{bmatrix},
\end{align}
where $E_{\alpha_r}$ is as in (\ref{Jblock})
and the diagonal matrix with alternating $\pm 1$ is:
\begin{equation*}
  F_{\alpha_r}:=\begin{bmatrix}
                 -1    &       &     & 0    \\
						           & 1     &     &     \\     
						           &       & \ddots &       \\
                 0     &       &     & (-1) ^{\alpha_r-1}  
                                   \end{bmatrix},
																										\qquad						(\alpha_r\textrm{-by-}\alpha_r),
\end{equation*} 
and $\Gamma_{\alpha_r}=(-1)^{\alpha_r} E_{\alpha_r} F_{\alpha_r}$.
By combining (\ref{Hr}) and (\ref{Ur}) with (\ref{UTHU}), we get
\[
(U_r^{-1})^T H_{r,j} (U_r^{-1})=\left\{
\begin{array}{ll}
(E_{\alpha_r} \oplus E_{\alpha_r})
\begin{bmatrix}
0 & -F_{\alpha_r}\\
 F_{\alpha_r} & 0 
\end{bmatrix}, & d + \alpha_r \textrm{ even}\\
E_{\alpha_r} (\theta_{r,j} (-1)^{\alpha} F_{\alpha_r}), & d + \alpha_r \textrm{ odd}
\end{array}
\right..
\]
Using this and recalling
\small
$\mu_r=\left\{
\begin{array}{ll}
2m_r,  & d + \alpha_r \textrm{ even} \\
m_r, & d + \alpha \textrm{ odd}
\end{array}
\right.$,
\normalsize
we can now write
\begin{align}
\label{U1TU1}
& (U^{-1})^{T} H U^{-1}  = ED,\\
E:=\bigoplus_{r=1}^{N}\big(\bigoplus_{j=1}^{\mu_r}  E_{\alpha_r}\big),  \quad
& D:=\bigoplus_{r=1}^{N}\big(\bigoplus_{j=1}^{m_r} 
D_{r,j} \big),
\,\,\,
D_{r,j}:=\left\{
\begin{array}{ll}
\hspace{-1mm} \begin{bsmallmatrix}
0 & -F_{\alpha_r}\\
F_{\alpha_r} & 0 
\end{bsmallmatrix}, & \hspace{-1mm} d + \alpha_r \textrm{ even}\\
\hspace{-1mm} \theta_{r,j} (-1)^{\alpha_r} F_{\alpha_r}, & \hspace{-1mm} d + \alpha_r \textrm{ odd}
\end{array}
\right.\hspace{-2mm}.\nonumber
\end{align}

B applyng (\ref{U1TU1}), the equation (\ref{QTQcase2}) transforms to
%
\begin{align}\label{QTQ2c}
E D = &  E  X^{T}(E D) X\\
D = & ( E  X^{T} E ) D X.\nonumber
\end{align}
We conjugate (\ref{QTQ2c}) with $\Omega
$ from Lemma \ref{lemaP}, and manipulate it:
\begin{align}\label{ortoD2}
 \Omega^{T} D\Omega = & \big(\Omega^{T}E\Omega \big)\big(\Omega^{T} X\Omega\big)^{T}\big(\Omega^{T}E\Omega\big)\big(\Omega^{T}D \Omega\big)\big(\Omega^{T} X\Omega\big)\\
  \mathcal{D} = &\mathcal{F}\mathcal{X}^{T}\mathcal{F}\mathcal{D}\mathcal{X}; \nonumber
\end{align}
\begin{align}
\label{eqD}
&\mathcal{D}:=  \Omega^{T} D \Omega=\bigoplus_{r=1}^{N} 
\big(
\bigoplus_{j=1}^{\alpha_r}
(-1)^{j-1}\widetilde{D}_{r}
\big),\quad
\widetilde{D}_{r}:=\left\{
\begin{array}{ll}
\hspace{-1mm}\bigoplus_{j=1}^{m_r}
\begin{bsmallmatrix}
0 & 1\\
-1 & 0 
\end{bsmallmatrix}, &  \hspace{-1mm} d + \alpha_r \textrm{ even}\\
\hspace{-1mm}(-1)^{\alpha_r} \Theta_{r}, & \hspace{-1mm} d + \alpha_r \textrm{ odd}
\end{array}
\right.\hspace{-3mm},\\
\label{eqXF}
& \mathcal{F}:= \Omega^{T}E\Omega=\bigoplus_{r=1}^{N} E_{\alpha_r}(I_{\mu_r}), \qquad \mathcal{X}:=\Omega^{T}X\Omega,
\end{align}
in which we denote $\Theta_{r}=\oplus_{j=1}^{m_r}\vartheta_{r,j}$.
In view of a special form of $X$ (obtained above),
we conclude that $\mathcal{X}$ is of the form (\ref{0T0}) for $\alpha=(\alpha_1,\ldots,\alpha_N)$, $\mu=(\mu_1,\ldots,\mu_N)$.

By conjugating with appropriate permutation matrices, we transform $\widetilde{D}_r$ for $d+\alpha_r$ even in (\ref{eqD}) to the standard symplectic matrices:
\begin{align*}
\mathcal{B}:= & \Phi^{T}\mathcal{D}\Phi=\bigoplus_{r=1}^{N} 
\big(
\bigoplus_{j=1}^{\alpha_r}
(-1)^{j-1}B_r
\big),\qquad
B_r:=\left\{
\begin{array}{ll}
\begin{bsmallmatrix}
0 & I_{m_r}\\
-I_{m_r} & 0 
\end{bsmallmatrix}, & d + \alpha_r \textrm{ even}\\
(-1)^{\alpha_r } \Theta_{m_r}, & d + \alpha_r \textrm{ odd}
\end{array}
\right.,\\
& \Phi=\bigoplus_{r=1}^{N} 
\big(
\bigoplus_{j=1}^{\alpha_r}
(-1)^{j-1}\Phi_r
\big),\qquad
\Phi_r:=\left\{
\begin{array}{ll}
\Omega_{2,m_r}, & d + \alpha_r \textrm{ even}\\
I_{m_r}, &  d + \alpha_r \textrm{ odd}
\end{array}
\right..
\end{align*}
Since $\Phi_r^{-1} E_{\alpha_r}(I_{\mu_r})\Phi_r=E_{\alpha_r}(I_{\mu_r})$,
the matrix $\mathcal{F}$ in (\ref{eqXF}) stays intact after conjugation by $\Phi$, so the equation (\ref{ortoD2}) becomes:
%
\begin{align}\label{ortoD2B}
	 \mathcal{B} = & \mathcal{F}\mathcal{Y}^T\mathcal{F}\mathcal{B}\mathcal{Y}, \qquad \quad \mathcal{Y}:=\Phi^T\mathcal{X}(\Phi^T)^{-1}. 
\end{align}
%
This equation is of the form (\ref{eqFYFIY}), with $\mathcal{F}$ and $\mathcal{B}=\mathcal{C}$ as in (\ref{BBF}), and $\mathcal{X}=\mathcal{Y}$ of the form (\ref{0T0}), all for $\alpha=(\alpha_1,\ldots,\alpha_N)$, $\mu=(\mu_1,\ldots,\mu_N)$.

Finally, $H$-Orthogonal or $H$-symplectic solutions of the equation $\mathcal{A}Q=Q\mathcal{A}$ are of the form:
\begin{equation}\label{psi0}
Q=U^{-1}\Omega (\Phi^T)^{-1}\mathcal{Y}\Phi^T\Omega^{T}U=\Psi^{-1}\mathcal{Y}\Psi, \qquad \Psi:=\Phi^T\Omega^{T}U,
\end{equation}
with $\mathcal{Y}$ as a solution of (\ref{ortoD2B}), which is provided by Lemma \ref{EqT}.
\end{proof}

The proof ot Theorem~\ref{stabw1} follows the same strategy as that of Theorem~\ref{stabw2}, but the presence of the $\lambda$ and $-\lambda$ Jordan blocks requires several additional nontrivial steps, starting with a suitable permutation of the Jordan blocks.

\begin{proof}[Proof of Theorem \ref{stabw1}]
Let $H$ and the $H$-skew-symmetric ($d=0$) or $H$-Hamiltonian ($d=1$) normal form $\mathcal{A}$, with one eigenvalue $\lambda\in \mathbb{C}\setminus\{0\}$, be:
\begin{align}
\label{H1l}
\mathcal{A}
=\bigoplus_{r=1}^{N}\Big( \bigoplus_{j=1}^{m_r}  \big( J_{\alpha_r}(\lambda)\oplus -(J_{\alpha_r}(\lambda))^T \big)\Big),\qquad 
H=\bigoplus_{r=1}^{N}\Big( \bigoplus_{j=1}^{m_r}  
\begin{bsmallmatrix}
0 & I_{\alpha_r}\\
(-1)^{d} I_{\alpha_r} & 0 
\end{bsmallmatrix} \Big).
\end{align}
Let $\Gamma_{\alpha_r}$ be as in (\ref{Jblock}), and denote by $\mathcal{J}(\mathcal{A})$ the Jordan canonical form of $\mathcal{A}^{}$ and by $U$ the corresponding transition matrix $U$ (i.e. $\mathcal{A}=U^{-1}\mathcal{J}(\mathcal{A}) U$):
\begin{align}\label{JCF1}
&\mathcal{J}(\mathcal{A})=\bigoplus_{r=1}^{N}\Big( \bigoplus_{j=1}^{m_r} \big( J_{\alpha_r}(\lambda)\oplus  J_{\alpha_r}(-\lambda)\big)\Big),\qquad 
U=\bigoplus_{r=1}^{N}\big( \bigoplus_{j=1}^{m_r} 
(I_{\alpha_r}\oplus  \Gamma_{\alpha_r})
\big).
\end{align}

We conjugate $\mathcal{J}(\mathcal{A})$ by an appropriate permutation matrix to collect together blocks corresponding to the same eigenvalue ($\lambda$ or $-\lambda$).
We use an analog of (\ref{perS}) (with $\alpha=2$) for block matrices. Define a $2m$-by-$2m$ block  permutation matrix whose blocks are identities and zero matrices:
\begin{equation}\label{perS2}
\Omega_{2,m}(I_{\alpha}):=\left[e_1(I_{\alpha})\;\,e_{3}(I_{\alpha})\;\,\ldots\;\,e_{2m-1}(I_{\alpha})\;\,e_2(I_{\alpha})\;\,e_{4}(I_{\alpha})\;\,\ldots\;\,e_{2m}(I_{\alpha})\right],
\end{equation}
in which $e_j(I_{\alpha})$ is a column ($n$ rows) with $I_{\alpha}$ in the $j$-th row and ${\alpha}$-by-${\alpha}$ zero matrices otherwise.
Post-multiplication by $\Omega_{2,m}(I_{\alpha})$ (with $\Omega_{2,m}^{T}$) puts the 1st, the 3rd, \ldots, the $(2m-1)$-th column, the 2nd, the 4th,\ldots, the $2m$-th (row) together. Thus:
\begin{align}\label{J1F}
\widetilde{\Omega}_1^{T}\mathcal{J}(\mathcal{A})\widetilde{\Omega}_1=\bigoplus_{r=1}^{N}\Big( \bigoplus_{j=1}^{m_r} J_{\alpha_r}(\lambda)\oplus \bigoplus_{j=1}^{m_r} J_{\alpha_r}(-\lambda)\Big), \quad \widetilde{\Omega}_1:=\bigoplus_{r=1}^{N} \Omega_{2,m_r}(I_{\alpha_r}).
\end{align}
Similarly as in (\ref{perS2}), we can permute blocks of possibly different dimensions:
\small
\begin{align}\label{perK2}
\widetilde{\Omega}_2=\left[e_1(I_{m_1\alpha_1})\;e_{3}(I_{m_2\alpha_2})\;\ldots\;e_{2N-1}(I_{m_N\alpha_N})\;e_2(I_{m_1\alpha_1})\;e_{4}(I_{m_2\alpha_2})\;\ldots\;e_{2N}(I_{m_N\alpha_N})
\right], 
\end{align}
\normalsize
where $e_j(I_{m_r\alpha_r})$ is a column ($2N$ rows) with $I_{m_r\alpha_r}$ in the $j$-th row and the appropriate zero matrices elsewhere.
Multiplication by $\widetilde{\Omega}_2$ (by $\widetilde{\Omega}_2^{T}$) from the right (left)
arranges columns (rows) in the same order as would $\Omega_{2,N}(I_{\alpha})$.
From (\ref{J1F}) we get
\begin{equation}\label{J2F}
\widetilde{\mathcal{J}}(\mathcal{A}):=\widetilde{\Omega}^{T}\mathcal{J}(\mathcal{A})\widetilde{\Omega}=
\bigoplus_{r=1}^{N}\big( \bigoplus_{j=1}^{m_r} J_{\alpha_r}(\lambda)\big)\oplus 
\bigoplus_{r=1}^{N}\big( \bigoplus_{j=1}^{m_r}  J_{\alpha_r}(-\lambda)\big), \quad 
\widetilde{\Omega}:=\widetilde{\Omega}_1\widetilde{\Omega}_2.
\end{equation}
%

By setting $\widetilde{U}=\widetilde{\Omega}^TU$, the equation $\mathcal{A}Q=Q\mathcal{A}$ (see (\ref{HQQH})) with $H$-orthogonal/$H$-symplectic $Q$, transforms to: 
\begin{equation}\label{wJQQwJ}
\widetilde{\mathcal{J}}(\mathcal{A})X=X\widetilde{\mathcal{J}}(\mathcal{A}), \qquad X=\widetilde{U} Q \widetilde{U}^{-1},\quad 
\qquad (\widetilde{U}=\widetilde{\Omega}^TU,\,\, Q^T H Q=H).
\end{equation}
Proposition \ref{resAoXXA} gives the solution $X=X_1\oplus X_2$ of (\ref{wJQQwJ}), where $X_1,X_2$ are $N$-by-$N$ block matrices such that their blocks $(X_1)_{rs}$, $(X_2)_{rs}$ are $m_r$-by-$m_s$ block matrices whose blocks are $\alpha_r$-by-$\times \alpha_s$ rectangular Toeplitz of the form (\ref{QTY}).

Using (\ref{QK2}) for $U=\widetilde{U}$, we describe $H$-orthogonality/$H$-sym\-ple\-cti\-city of $Q$ by:
%
\begin{align}\label{QTQHH}
\widetilde{\Omega}^{T}(U^{-1})^{T} H U^{-1}\widetilde{\Omega} =   &  X^T \big(\widetilde{\Omega}^{T}(U^{-1})^{T} H U^{-1} \widetilde{\Omega}\big) X, \qquad \quad
(Q=\widetilde{U}^{-1} X \widetilde{U}).
\end{align}
We again use $\Gamma_{\alpha}^{-1}=(-1)^{\alpha +1}\Gamma_{\alpha}^{}$, but in contrast to (\ref{UTHU}), we now derive 
\begin{align}\label{s2aH}
( (I_{\alpha} \oplus \Gamma_{\alpha})^{-1})^T\begin{bmatrix}
0 & I_{\alpha} \\
(-1)^{d} I_{\alpha} & 0
\end{bmatrix}(I_{\alpha} \oplus \Gamma_{\alpha})^{-1}
= &
\begin{bmatrix}
0 & (-1)^{\alpha-1}\Gamma_{\alpha}\\
(-1)^{d+\alpha-1} \Gamma_{\alpha}^{T} & 0 
\end{bmatrix}
=
\\
=
\big((-1)^{\alpha-1}\Gamma_{\alpha}\oplus (-1)^{d+\alpha-1}\Gamma_{\alpha}^T\big)
\begin{bmatrix}
0 & I_{\alpha}\\
I_{\alpha} & 0 
\end{bmatrix}.  \nonumber
\end{align}
Using (\ref{s2aH}) we now decompose $(U^{-1})^{T} H U^{-1} $ (the trick of the proof):
\begin{align}\label{deDVH}
&(U^{-1})^{T} H U^{-1}\hspace{-1mm}
=\bigoplus_{r=1}^{N}\left(\bigoplus_{j=1}^{m_r} 
\Big(
\big((-1)^{\alpha_r-1}\Gamma_{\alpha_r}\oplus (-1)^{d+\alpha_r-1}\Gamma_{\alpha_r}^T\big)
\begin{bmatrix}
0 & I_{\alpha_r}\\
I_{\alpha_r} & 0 
\end{bmatrix}
\Big)\right)
=D K\\
& D:=\bigoplus_{r=1}^{N}\Big(\bigoplus_{j=1}^{m_r}  \big((-1)^{\alpha_r-1}\Gamma_{\alpha_r}\oplus (-1)^{d+\alpha_r-1}\Gamma_{\alpha_r}^T \big)  \Big), \qquad 
K: =\bigoplus_{r=1}^{N}\big(\bigoplus_{j=1}^{m_r}  
\begin{bsmallmatrix}
0 & I_{\alpha_r}\\
I_{\alpha_r} & 0 
\end{bsmallmatrix}
 \big).
\end{align}
Furthermore, (\ref{deDVH}) yields the decomposition of $\widetilde{\Omega}^{T} \big((U^{-1})^{T} H U^{-1}\big) \widetilde{\Omega}$:
\begin{align}
\label{BK}
&\widetilde{\Omega}^{T} \big((U^{-1})^{T} H U^{-1}\big) \widetilde{\Omega}
=(\widetilde{\Omega}^{T} D \widetilde{\Omega})(\widetilde{\Omega}^{T} K \widetilde{\Omega})=B \widetilde{K},\nonumber\\
&B:=\widetilde{\Omega}^{T} D \widetilde{\Omega}
=B_1 \oplus (-1)^{d} B_1^{T}, \,\,\, B_1:=\bigoplus_{r=1}^{N}\big(\bigoplus_{j=1}^{m_r} (-1)^{\alpha_r-1} \Gamma_{\alpha_r}  \big),\\
&\widetilde{K}:=\widetilde{\Omega}^{T} V \widetilde{\Omega}  =
\begin{bmatrix}
0 & I_{n}\\
I_{n} & 0 
\end{bmatrix},\qquad  n:=\sum_{r=1}^{N}\alpha_r m_r.\nonumber
\end{align}
By applying (\ref{BK}) to (\ref{QTQHH}) we obtain:
\begin{align}\label{QTQH2}
B \widetilde{K} = &  X^{T}  B \widetilde{K} X \nonumber\\
 B_1 \oplus (-1)^{d} B_1^{T} = & (X_1^{T} \oplus X_2^{T}) ( B_1 \oplus (-1)^{d} B_1^{T}) (\widetilde{K} (X_1 \oplus X_2) \widetilde{K}^{-1})\\
B_1 \oplus  (-1)^{d} B_1^{T} = & (X_1^{T} \oplus X_2^{T}) (  B_1 \oplus (-1)^{d} B_1^{T}) (X_2 \oplus X_1)\nonumber\\
B_1 = &  X^{T}_1 B_1 X_2,\qquad \quad \big( (-1)^{d} B_1^{T}=X^{T}_2 (\,(-1)^{d} B_1^{T}) X_1 \,\big). \nonumber
\end{align}

By conjugating both hand sides of the last equation of (\ref{QTQH2}) by $\Omega$ from Lemma \ref{lemaP}, and then slightly manipulating it,
we deduce
\begin{align}\label{ortoD3}
  \Omega^{T}B_1\Omega = & (\Omega^{T} X_1\Omega)^{T}(\Omega^{T}B_1\Omega)(\Omega^{T} X_2\Omega)\\
\mathcal{B}_1= &  \mathcal{X}^{T}_1\mathcal{B}_1 \mathcal{X}_2 \qquad \qquad (\mathcal{X}_1:=\Omega^{T} X_1\Omega, \quad \mathcal{X}_2:=\Omega^{T} X_2\Omega),\nonumber
\end{align}
where $\mathcal{B}_1:= \bigoplus_{r=1}^{N}\big( 
(-1)^{\alpha_r-1} \Gamma_{\alpha}(I_{m_r})
\big)$, in which 
$
\Gamma_{\alpha_r}(I_{m_r})
$ is a block exchange matrix with altermating $\pm I_{m_r}$ on the anti-diagonal (see (\ref{Jblock})).
We observe that the form of the solution $X=X_1\oplus X_2$ of (\ref{wJQQwJ}) yields $\mathcal{X}_1, \mathcal{X}_2\in \mathbb{T}^{\alpha,\mu}$; cf. (\ref{0T0}). 
Therefore, $H$-orthogonal or $H$-symplectic solution $Q$
of $\mathcal{A}Q=Q\mathcal{A}$ is of the form
\begin{align}\label{QPsi}
&Q
=\Psi^{-1} \big(\mathcal{X}_1\oplus (\mathcal{X}^{T}_1)^{-1}\big)\Psi, \qquad \Psi:=(I\oplus \mathcal{B}_1)\big((\Omega^{T}\oplus \Omega^{T})\widetilde{\Omega}^{T} U\big), \quad \mathcal{X}_1\in \mathbb{T}^{\alpha,\mu}.\\
\end{align}
This completes the proof of the theorem.
\end{proof}

To prove Theorems \ref{stabw1Ri}-\ref{stabw2Ri}, we transform the real canonical forms to the complex canonical forms of Theorem~\ref{stabw1}, apply that theorem, and then impose the reality condition on the resulting group.

\begin{proof}[Proof of Theorem \ref{stabw1Ri}]
For $d=0$ and $d=1$, respectively, consider real $\widetilde{J}$-skew-symmetric and $\widetilde{J}$-Hamiltonian normal form, with eigenvalues $\pm ib$ for $b\in  \mathbb{R}\setminus \{0\}$:
\begin{align}
\label{H1lr}
&\mathcal{A}
=\bigoplus_{r=1}^{N}\Big( \bigoplus_{j=1}^{m_r}   J_{\alpha_r}^{\mathbb{R}}(0,b) \Big),\qquad 
\widetilde{H}=\bigoplus_{r=1}^{N}\Big( \bigoplus_{j=1}^{m_r} 
\vartheta_{r,j} \Delta_{\alpha_r,d}\Big),\qquad 
\bigoplus_{j=1}^{m_r}\vartheta_{r,j}=I_{p_r,q_r},
\end{align}
where $\Delta_{\alpha_r,d}$ and $J_{\alpha_r}^{\mathbb{R}}(0,b) $ are defined by (\ref{Dmd}) and (\ref{JblockR}), respectively.
Using 
%
\begin{align}
S=\bigoplus_{r=1}^{N}\Big( \bigoplus_{j=1}^{m_r} 
\Omega_{2,\alpha_r} 
\Big(
\frac{1}{\sqrt{2}}
\begin{bsmallmatrix}
I_{\alpha_r} & iI_{\alpha_r}\\
iI_{\alpha_r} & I_{\alpha_r}
\end{bsmallmatrix}
\Big) (I_{\alpha_r}\oplus \Gamma_{\alpha_r})
\Big), \quad
D=\bigoplus_{r=1}^{N}\Big( \bigoplus_{j=1}^{m_r} 
\sqrt{\vartheta_{r,j} \delta_r}\,\, I_{2\alpha_r}
\Big),
\end{align}
the real canonical form $\mathcal{A}$ is transformed to the complex 
canonical form: 
\begin{align*}
&\mathcal{A}'
=S^{-1} \mathcal{A} S 
= D^{-1} (S^{-1} \mathcal{A} S) D
=\bigoplus_{r=1}^{N}\Big( \bigoplus_{j=1}^{m_r}  \big( \big(J_{\alpha_r}(i b)  \oplus -\big( J_{\alpha_r}(i b) )\big)^T \big)\Big), 
\\
& H'=  
S^{T} H S=\bigoplus_{r=1}^{N}
\delta_r
\Big( \bigoplus_{j=1}^{m_r}  
\vartheta_{r,j} 
\begin{bsmallmatrix}
0 &  I_{\alpha_r}\\
(-1)^{d} I_{\alpha_r} & 0 
\end{bsmallmatrix}
\Big), \quad
\delta_r:=
\left\{
\begin{array}{ll}
\hspace{-1mm}i^{\alpha_r+d},& \alpha_r+d \textrm{ odd}\\
\hspace{-1mm}i^{\alpha+d-1},& \alpha_r+d \textrm{ even}
\end{array}
\right.\hspace{-2mm},\\
& 
H'' = D^{T} H' D= D^{T} (S^{T} H S ) D= \bigoplus_{r=1}^{N}\Big( \bigoplus_{j=1}^{m_r} 
\begin{bsmallmatrix}
0 &  I_{\alpha_r}\\
(-1)^{d} I_{\alpha_r} & 0 
\end{bsmallmatrix} 
\Big).
\end{align*}
We now apply Theorem \ref{stabw1} (the complex case) to $\mathcal{A}'$ with $H''$; see also (\ref{QPsi}). We obtain $H''$-orthogonal/$H''$-symplectic solution $Q'$ of the equation $\mathcal{A}'Q'=Q' \mathcal{A}'$:
\begin{align*}
Q'=& \Psi^{-1} \big(\mathcal{X}\oplus (\mathcal{X}^{T})^{-1}\big)\Psi,\quad \Psi:=(I\oplus \mathcal{B}_1)\big((\Omega^{T}\oplus \Omega^{T})\widetilde{\Omega}^{T}U\big), \quad \mathcal{X}\in \mathbb{T}^{\alpha,\mu},
\end{align*}
where $U$,$\widetilde{\Omega}$, $\mathcal{B}_1$ are as in (\ref{JCF1}), (\ref{J2F}), (\ref{BK}), respectively, and $\Omega$ is from Lemma \ref{lemaP}.
Moreover, an $H$-orthogonal/$H$-symplectic solution of $\mathcal{A} Q= Q\mathcal{A}$ is:
\begin{align}
\label{QHA}
Q:= (S D) Q' (D^{-1}S^{-1}) = (S D \Psi^{-1})\mathcal{Y}(\Psi D^{-1}S^{-1}), \qquad \mathcal{Y}:=\mathcal{X}\oplus (\mathcal{X}^{T})^{-1}. 
\end{align}
We now make two crucial observations. First, when inspecting
\begin{align}
\label{conj1}
\Psi D^{-1} S^{-1}=\big((I\oplus \mathcal{B}_1)(\Omega^{T}\oplus \Omega^{T})\widetilde{\Omega}^{T} U \big) D^{-1} S^{-1},
\end{align}
we notice that
\begin{align} 
\label{conj2}
U D^{-1} S^{-1}=  &
\bigoplus_{r=1}^{N}\Big( \bigoplus_{j=1}^{m_r} 
\Big(
\frac{1}{\sqrt{\vartheta_{r,j} \delta_r}}
\frac{1}{\sqrt{2}} 
\begin{bsmallmatrix}
I_{\alpha_r} & -iI_{\alpha_r}\\
-iI_{\alpha_r} & I_{\alpha_r}
\end{bsmallmatrix}
\Omega_{2,\alpha_r} ^T
\Big)
= P D^{-1} R^T,\\
P:= &
\bigoplus_{r=1}^{N}\Big( \bigoplus_{j=1}^{m_r} 
\frac{1}{\sqrt{2}}
\begin{bsmallmatrix}
I_{\alpha_r} & -iI_{\alpha_r}\\
-iI_{\alpha_r} & I_{\alpha_r}
\end{bsmallmatrix} 
\Big),
\qquad
R:=\bigoplus_{r=1}^{N}\Big( \bigoplus_{j=1}^{m_r} 
\Omega_{2,\alpha_r}
\Big).\nonumber
\end{align}
Second, we apply (\ref{conj1}) and (\ref{conj2}) to the real conjugate of $Q$ in (\ref{QHA}) (via $R\widetilde{\Omega}^{}(\Omega^{} \oplus \Omega^{})$), 
to observe:
\begin{align}
\label{QwR2}
\widetilde{Q}:=& \big(R \widetilde{\Omega}(\Omega \oplus \Omega)\big)^{-1} Q \big(R\widetilde{\Omega}(\Omega \oplus \Omega)\big)\\
=  &\big( (\Omega^T \oplus \Omega^T)\widetilde{\Omega}^T R^T S D \Psi^{-1}\big) \mathcal{Y} \big(\Psi D^{-1} S^{-1}R \widetilde{\Omega}(\Omega \oplus \Omega)\big), \nonumber
\end{align} 
\normalsize
\begin{align*}
\Psi D^{-1} S^{-1}R \widetilde{\Omega}(\Omega \oplus \Omega)  & = (I\oplus \mathcal{B}_1)(\Omega^{T}\oplus \Omega^{T})\widetilde{\Omega}^{T}U D^{-1} S^{-1}R(\widetilde{\Omega}(\Omega \oplus \Omega)\\
&= (I\oplus \mathcal{B}_1)(\Omega^{T}\oplus \Omega^{T})\widetilde{\Omega}^{T}P D^{-1} (R^T R)\widetilde{\Omega}(\Omega \oplus \Omega)\\
&=  (I\oplus \mathcal{B}_1)\mathcal{P} (\mathcal{D}^{-1}\oplus \mathcal{D}^{-1}),
\end{align*}
%
\begin{align*}
&\mathcal{P}:=(\Omega^{T}\oplus \Omega^{T})\widetilde{\Omega}^{T}P \widetilde{\Omega}(\Omega \oplus \Omega)
=
\frac{1}{\sqrt{2}}
\begin{bsmallmatrix}
I_n &  -iI_{n}\\
-iI_{n} & I_n 
\end{bsmallmatrix},\\
&\mathcal{D}\oplus \mathcal{D}=(\Omega^{T}\oplus \Omega^{T})(\widetilde{\Omega}^{T} D \widetilde{\Omega})(\Omega \oplus \Omega), \quad  \mathcal{D}:= \bigoplus_{r=1}^{N}\delta_r\Big( \bigoplus_{j=1}^{\alpha_r} \big( \bigoplus_{k=1}^{m_r} \sqrt{\vartheta_{r,k} }\big) \Big).
\end{align*}

Using this, the matrix $\widetilde{Q}$ in (\ref{QwR2}), with $\mathcal{Y}=\mathcal{X}\oplus (\mathcal{X}^{T})^{-1}$, can be rewritten as
\begin{align*}
\widetilde{Q} & = (\mathcal{D} \oplus \mathcal{D} )\mathcal{P}^{-1}   (I\oplus \mathcal{B}_1^{-1})\big(\mathcal{X}\oplus (\mathcal{X}^{T})^{-1}\big)(I\oplus \mathcal{B}_1) \mathcal{P}( \mathcal{D}^{-1} \oplus \mathcal{D}^{-1})\\
&=
\begin{bmatrix}
\mathcal{D} \mathcal{X}\mathcal{D}^{-1}+\mathcal{D} \mathcal{B}^{-1}(\mathcal{X}^T)^{-1}\mathcal{B}\mathcal{D} ^{-1} &  -i(\mathcal{D} \mathcal{X} \mathcal{D} ^{-1}-\mathcal{D} \mathcal{B}^{-1}(\mathcal{X}^T)^{-1}\mathcal{B}\mathcal{D} ^{-1})\\
i(\mathcal{D} \mathcal{X}\mathcal{D}^{-1}-\mathcal{D} \mathcal{B}^{-1}(\mathcal{X}^T)^{-1}\mathcal{B}\mathcal{D}^{-1}) & \mathcal{D} \mathcal{X}\mathcal{D}^{-1}+\mathcal{B}^{-1}(\mathcal{X}^T)^{-1}\mathcal{B} \mathcal{D}^{-1}
\end{bmatrix}.
\end{align*}
Therefore, $\widetilde{Q}$ (and hence $Q$ by (\ref{QwR2})) is real if and only if $\mathcal{D}\mathcal{X}\mathcal{D}^{-1}+\mathcal{D}\mathcal{B}^{-1}(\mathcal{X}^T)^{-1}\mathcal{B}\mathcal{D}^{-1}$ is real, and $\mathcal{D}\mathcal{X}\mathcal{D}^{-1}-\mathcal{D}\mathcal{B}^{-1}(\mathcal{X}^T)^{-1}\mathcal{B}\mathcal{D}^{-1}$ is purely imaginary. Hence
\begin{align*}
\mathcal{D}\mathcal{X}\mathcal{D}^{-1}  = &    \overline{\mathcal{D}}\mathcal{B}^{-1} (\overline{\mathcal{X}}^T)^{-1} \mathcal{B}\overline{\mathcal{D}}^{-1} \\[-10pt] 
\mathcal{F}\overline{\mathcal{X}}^{-1}\mathcal{F} (\mathcal{F}\mathcal{B}\overline{\mathcal{D}}^{-1}\mathcal{D})\mathcal{X} = &  (\mathcal{F}\mathcal{B}\overline{\mathcal{D}}^{-1}\mathcal{D}), \qquad \mathcal{C}:=\mathcal{F}\mathcal{B}\overline{\mathcal{D}}^{-1}\mathcal{D}=\bigoplus_{r=1}^{N}\Big( \bigoplus_{j=1}^{\alpha_r}(-1)^{j-1}\delta_r   I_{p_r,q_r}\Big) \\[-10pt]
 \mathcal{F}\mathcal{X}^{*}\mathcal{F} \mathcal{C}\mathcal{X} = & \mathcal{C}.
\end{align*}
To conclude the proof we apply Lemma \ref{EqT} for $\star=*$, $\varepsilon=-1$ and $B_r=\delta_r   I_{p_r,q_r}$.
\end{proof}

\begin{proof}[Proof of Theorem \ref{stabw2Ri}]
For $d=0$ and $d=1$, respectively, consider real $\widetilde{J}$-skew-symmetric and $\widetilde{J}$-Hamiltonian normal form, with eigenvalues $\lambda:=a+ib\in \mathbb{C}\setminus (\mathbb{R}\cup i\mathbb{R})$:
\begin{align}
\label{H1lr}
\mathcal{A}
=\bigoplus_{r=1}^{N}\Big( \bigoplus_{j=1}^{m_r}  \big(J_{\alpha_r}^{\mathbb{R}}(a,b)\oplus -(J_{\alpha_r}^{\mathbb{R}}(a,b))^T \big)\Big),\quad 
H=\bigoplus_{r=1}^{N}\Big( \bigoplus_{j=1}^{m_r}  
\big(\begin{bsmallmatrix}
0 & I_{2\alpha_r}\\
(-1)^{d} I_{2\alpha_r} & 0 
\end{bsmallmatrix} \big)\Big).
\end{align}
Via the appropriate transition matrix $P$, with $S^T H S=H$, we transform it to complex $H$-skew-symmetric/$H$-Hamiltonian normal form:
\begin{align}
\label{H2r}
&\mathcal{A}'
=\bigoplus_{r=1}^{N}\Big( \bigoplus_{j=1}^{m_r}  \big( \big(J_{\alpha_r}(\lambda) \oplus J_{\alpha_r}(\overline{\lambda})\big) \oplus -\big(J_{\alpha_r}(\lambda)\oplus J_{\alpha_r}(\overline{\lambda})\big)^T \big)\Big), \quad (\mathcal{A}' = S^{-1}\mathcal{A} S) \\
& S =\bigoplus_{r=1}^{N}\Big( \bigoplus_{j=1}^{m_r} 
\Big(
\big(\Omega_{2,\alpha_r}\oplus \Omega_{2,\alpha_r}\big)
\frac{1}{\sqrt{2}}\big(\begin{bmatrix}
I_{\alpha_r} & iI_{\alpha_r}\\
iI_{\alpha_r} & I_{\alpha_r}
\end{bmatrix}
\oplus
\begin{bmatrix}
I_{\alpha_r} & -iI_{\alpha_r}\\
-iI_{\alpha_r} & I_{\alpha_r}
\end{bmatrix}
\Big)
\Big).
\end{align}
%
Similarly, as we obtained (\ref{J2F}) from (\ref{JCF1}), we use a suitable permutation matrix to collect together blocks of $\mathcal{A}'$ in (\ref{H2r}), corresponding to eigenvalues $\pm \lambda$ and $\pm \overline{\lambda}$:
\small
\begin{align*}
\widetilde{\mathcal{A}}'
= & (\widetilde{\Omega}')^{-1}\mathcal{A}' \widetilde{\Omega}'
=\bigoplus_{r=1}^{N}\Big( \bigoplus_{j=1}^{m_r}  \big( \big(J_{\alpha_r}(\lambda) \oplus -(J_{\alpha_r}(\lambda))^{T} \big)\Big) \oplus \bigoplus_{r=1}^{N}\Big( \bigoplus_{j=1}^{m_r}  \big(J_{\alpha_r}(\overline{\lambda})  \oplus -(J_{\alpha_r}(\overline{\lambda})^{T}\big) \big)\Big),\\
H'
= & (\widetilde{\Omega}')^{T} H \widetilde{\Omega}'
= 
\bigoplus_{r=1}^{N}\bigoplus_{j=1}^{m_r} \begin{bsmallmatrix}
0 & I_{\alpha_r}\\
(-1)^{d}I_{\alpha_r} & 0 
\end{bsmallmatrix} \oplus
\bigoplus_{r=1}^{N}\bigoplus_{j=1}^{m_r} \begin{bsmallmatrix}
0 & I_{\alpha_r}\\
(-1)^{d}I_{\alpha_r} & 0 
\end{bsmallmatrix},\\
&\qquad \widetilde{\Omega}':=  \widetilde{\Omega}_1'\widetilde{\Omega}_2', \qquad \widetilde{\Omega}_1':=\bigoplus_{r=1}^{N}
 \Omega_{2,2m_r}(I_{\alpha_r}),\\ 
%
&\widetilde{\Omega}_2':=\left[e_1(I_{2m_1\alpha_1})\;e_{3}(I_{2m_2\alpha_2})\;\ldots\;e_{2N-1}(I_{2m_N\alpha_N})\;e_2(I_{2m_1\alpha_1})\;e_{4}(I_{2m_2\alpha_2})\;\ldots\;e_{2N}(I_{2m_N\alpha_N})
\right].
\end{align*}
\normalsize
%

Applying Theorem \ref{stabw1} with (\ref{QPsi}) (the complex case) to $\widetilde{\mathcal{A}}'$, we obtain $H'$-orthogonal/$H'$-symplectic solution $Q'$ of the equation $\widetilde{\mathcal{A}}'Q'=Q' \widetilde{\mathcal{A}}'$:
%
\begin{align*}
Q'=&Q_1' \oplus Q_2'=  (\Psi^{-1}\mathcal{Y}_1\Psi)\oplus (\Psi^{-1} \mathcal{Y}_2 \Psi), \qquad \Psi:=(I\oplus \mathcal{B}_1)(\Omega^{T}\oplus \Omega^{T})\widetilde{\Omega}^{T}U\\
&\mathcal{Y}_1:=\mathcal{X}_1\oplus (\mathcal{X}^{T}_1)^{-1}, \quad \mathcal{Y}_2 :=\mathcal{X}_2\oplus (\mathcal{X}^{T}_2)^{-1}, \qquad 
\mathcal{X}_1,\mathcal{X}_2\in \mathbb{T}^{\alpha,\mu},
\end{align*}
where $U$,$\widetilde{\Omega}$, $\mathcal{B}_1$ are as in (\ref{JCF1}), (\ref{J2F}), (\ref{BK}), respectively, and $\Omega$ is from Lemma \ref{lemaP}.
Moreover, $H$-orthogonal/$H$-symplectic solution $Q=\big(S \widetilde{\Omega}'\big) Q' \big((\widetilde{\Omega}')^{-1} S^{-1} \big)$, of $\mathcal{A}Q=Q\mathcal{A}$, 
is real precisely when its real conjugate, via $R^{}\widetilde{\Omega}'$, is real:
%
\begin{align}
\label{wQR}
(R\widetilde{\Omega}')^{-1} Q (R^{}\widetilde{\Omega}')
= &
\big((\widetilde{\Omega}')^{T} R^{T} S^{} \widetilde{\Omega}'\big) Q'  \big((\widetilde{\Omega}')^{-1} S^{-1} R^{}\widetilde{\Omega}'\big), \qquad 
n:=\sum_{r=1}^{N}2\alpha_r m_r,\\
= & 
\frac{1}{2}
\begin{bsmallmatrix}
I_n & -iZ_{}\\
-iZ_{} & I_n 
\end{bsmallmatrix}  
(Q_1'\oplus Q_2')
\begin{bsmallmatrix}
I_n & iZ_{}\\
iZ_{} & I_n 
\end{bsmallmatrix}
=
\begin{bsmallmatrix}
Q_1+ZQ_2 Z & i(Q_1 Z_{}-ZQ_2)\\
i(Q_2 Z_{}- Z Q_1) & Z Q_1 Z + Q_2
\end{bsmallmatrix},\nonumber
\end{align}
\begin{align*}
&(\widetilde{\Omega}')^{T}   S^{-1} R \widetilde{\Omega}'
=
\frac{1}{\sqrt{2}}\begin{bsmallmatrix}
I_n & iZ_{}\\
iZ_{} & I_n 
\end{bsmallmatrix}, \,\,
R:=\bigoplus_{r=1}^{N}\Big(\hspace{-1mm} \bigoplus_{j=1}^{m_r}\big(\Omega_{2,\alpha_r}\oplus \Omega_{2,\alpha_r}\big)
\Big), \,
Z:=\bigoplus_{r=1}^{N}\bigoplus_{k=1}^{m_r} (-I_{\alpha_r}\oplus I_{\alpha_r}).
\end{align*}
Furthermore, (\ref{wQR}) is real if and only if $Q_2=Z \overline{Q}_1 Z$. 
Using $\widetilde{\Omega}$ from (\ref{J2F}), we consider an equivalent equation:
\begin{equation}
\label{q2q1b}
\widetilde{\Omega}^T Q_2 \widetilde{\Omega}=(\widetilde{\Omega}^T Z \widetilde{\Omega})(\widetilde{\Omega}^T \overline{Q}_1 \widetilde{\Omega} ) (\widetilde{\Omega}^T Z \widetilde{\Omega}), \qquad \qquad (\widetilde{\Omega}^T Z \Omega=-I_n\oplus I_n).
\end{equation}
Since
$\widetilde{\Omega}^{T}U\widetilde{\Omega}=
I_n \oplus \bigoplus_{r=1}^{N}\big(\bigoplus_{k=1}^{m_r} \Gamma_{\alpha_r}\big)
=I_n \oplus B_1^T$ with $B_1^TB_1=I_n$ (see (\ref{BK})), we get
\small
\begin{align*}
\Psi\widetilde{\Omega} & = \big((I\oplus \mathcal{B})(\Omega^{T}\oplus \Omega^{T})\widetilde{\Omega}^{T}U\big)\widetilde{\Omega}
=
(\Omega^{T}\oplus \Omega^{T}) \big((\Omega^{}\oplus \Omega^{})(I\oplus \mathcal{B}^{T})(\Omega^{T}\oplus \Omega^{T})\big)(\widetilde{\Omega}^{T}U\widetilde{\Omega})\\
& =  (\Omega^{T}\oplus \Omega^{T})
(I_n \oplus B_1) (I_n \oplus B_1^{T}) 
=\Omega^{T}\oplus \Omega^{T},
\end{align*}
\normalsize
therefore, we have:
\begin{align*}
\widetilde{\Omega}^T Q_k \widetilde{\Omega}
=  (\Psi' \widetilde{\Omega} )^{-1}  \mathcal{Y}_k (\Psi'\widetilde{\Omega})
= (\Omega^{}\oplus \Omega^{})\mathcal{Y}_k (\Omega^{T}\oplus \Omega^{T}), \qquad k\in \{1,2\}.
\end{align*}
\normalsize
%
The equation (\ref{q2q1b}) then splits into a system of two (equivalent) equations 
\[
\Omega^{} \mathcal{X}_2 \Omega^{T}= -\Omega^{}  \overline{\mathcal{X}}_1 (-\Omega^{T}) , \qquad 
\Omega  (\mathcal{X}_2^T)^{-1} \Omega^{T}=  \Omega  (\overline{\mathcal{X}}_1^T)^{-1}\Omega^{T}.
\]
It yields $\mathcal{X}_2=\overline{\mathcal{X}}_1$, which is equivalent to the fact that $Q$ is real.
\end{proof}

\begin{proof}[Proof of Theorem \ref{stabwU}]
Given a $H$-Hermitian normal form $\mathcal{A}$, we solve $\mathcal{A}Q=Q\mathcal{A}$ with $H$-unitary $Q$.
\begin{enumerate}[label={\bf Case \Roman*.},ref={Case \Roman*},wide=0pt,itemsep=15pt]
\item 
Suppose
\begin{align}
\mathcal{A}
=\bigoplus_{r=1}^{N}\Big( \bigoplus_{j=1}^{m_r}  \big( J_{\alpha_r}(\lambda)\oplus J_{\alpha_r}(\overline{\lambda}) \big)\Big),\qquad 
H=\bigoplus_{r=1}^{N}\Big( \bigoplus_{j=1}^{m_r} E_{2\alpha_r} \Big). 
\end{align}

We conjugate $\mathcal{A}$ by an appropriate permutation matrix to collect together blocks corresponding to the same eigenvalue ($\lambda$ or $\overline{\lambda}$).
We use $\widetilde{\Omega}$ from (\ref{J2F}) to get
\begin{equation}\label{J2FU}
\widetilde{\mathcal{J}}(\mathcal{A}):=\widetilde{\Omega}^{T}\mathcal{A}\widetilde{\Omega}=
\bigoplus_{r=1}^{N}\big( \bigoplus_{j=1}^{m_r} J_{\alpha_r}(\lambda)\big)\oplus 
\bigoplus_{r=1}^{N}\big( \bigoplus_{j=1}^{m_r}  J_{\alpha_r}(\overline{\lambda})\big), \quad 
\end{equation}

Thus $\mathcal{A}Q=Q\mathcal{A}$ (see (\ref{HQQH})) with $H$-unitary $Q$, transforms to 
\begin{equation}\label{wJQQwJU}
\widetilde{\mathcal{J}}(\mathcal{A})X=X\widetilde{\mathcal{J}}(\mathcal{A}), \qquad X=\widetilde{\Omega}^T Q \widetilde{\Omega},\quad 
\qquad (Q^T H Q=H).
\end{equation}
Proposition \ref{resAoXXA} gives the solution $X=X_1\oplus X_2$ of (\ref{wJQQwJ}), where $X_1,X_2$ are $N$-by-$N$ block matrices such that their blocks $(X_1)_{rs}$, $(X_2)_{rs}$ are $m_r$-by-$m_s$ block matrices whose blocks are $\alpha_r$-by-$\times \alpha_s$ rectangular Toeplitz of the form (\ref{QTY}).

\quad
The trick of this proof is to write
%
%
%
\begin{align*}
H=D K \qquad
D:=\bigoplus_{r=1}^{N}\Big(\bigoplus_{j=1}^{m_r}  \big(E_{\alpha_r}\oplus E_{\alpha_r} \big)  \Big), \quad 
K: =\bigoplus_{r=1}^{N}\big(\bigoplus_{j=1}^{m_r}  
\begin{bsmallmatrix}
0 & I_{\alpha_r}\\
I_{\alpha_r} & 0 
\end{bsmallmatrix}
 \big).
\end{align*}
and then to get the decomposition of $\widetilde{\Omega}^{T}  H  \widetilde{\Omega}$:
\begin{align}
\label{BKU}
&\widetilde{\Omega}^{T}  H \widetilde{\Omega}
=(\widetilde{\Omega}^{T} D \widetilde{\Omega})(\widetilde{\Omega}^{T} K \widetilde{\Omega})=B \widetilde{K},\\
B:=\widetilde{\Omega}^{T} D \widetilde{\Omega}
=E \oplus E, \quad & E:=\bigoplus_{r=1}^{N}\big(\bigoplus_{j=1}^{m_r} E_{\alpha_r}  \big), \quad
\widetilde{K}:=\widetilde{\Omega}^{T} V \widetilde{\Omega}  =
\begin{bmatrix}
0 & I_{n}\\
I_{n} & 0 
\end{bmatrix},\,  n:=\sum_{r=1}^{N}\alpha_r m_r.\nonumber
\end{align}

In view of (\ref{QK2}) for $U=\widetilde{\Omega}^T$, we then characterize the $H$-unitarity of $Q$ by:
%
\begin{align}
\label{QTQH2U}
\widetilde{\Omega}^{T} H \widetilde{\Omega} =   &  X^{*} \big(\widetilde{\Omega}^{T}H \widetilde{\Omega}\big) X, \qquad \quad
(Q=\widetilde{\Omega} X \widetilde{\Omega}^T)\nonumber\\
%
%
B \widetilde{K} = &  X^{*}  B \widetilde{K} X \nonumber\\
E \oplus E = & (X_1^{*} \oplus X_2^{*}) ( E \oplus E) (\widetilde{K} (X_1 \oplus X_2) \widetilde{K}^{-1})\\
E \oplus E = & (X_1^{*} \oplus X_2^{*}) ( E \oplus E) (X_2 \oplus X_1)\nonumber\\
E = &  X^{*}_1 E X_2,\qquad \quad \big( E = X^{*}_2 E X_1 \big). \nonumber
\end{align}

By conjugating both hand sides of the last equation of (\ref{QTQH2}) by $\Omega$ from Lemma \ref{lemaP}, recalling $\mathcal{F}=\Omega^{T}E\Omega$, and slightly manipulating it, we deduce
\begin{align}\label{ortoD3U}
  \Omega^{T}E \Omega = & (\Omega^{T} X_1\Omega)^{*}(\Omega^{T}E\Omega)(\Omega^{T} X_2\Omega)\\
\mathcal{F}= &  \mathcal{X}^{*}_1\mathcal{F} \mathcal{X}_2 \qquad \qquad (\mathcal{X}_1:=\Omega^{T} X_1\Omega, \quad \mathcal{X}_2:=\Omega^{T} X_2\Omega),\nonumber
\end{align}
The form of the solution $X=X_1\oplus X_2$ of (\ref{wJQQwJ}) yields $\mathcal{X}_1, \mathcal{X}_2\in \mathbb{T}^{\alpha,\mu}$ (cf. (\ref{0T0}).
Therefore, 
$Q$
is of the form
\begin{align}\label{QPsiU}
&Q=\widetilde{\Omega}(X_1\oplus X_2)\widetilde{\Omega}^{T}
=\Psi^{-1} \big(\mathcal{X}_1\oplus (\mathcal{X}^{*}_1)^{-1}\big)\Psi, \qquad \mathcal{X}_1\in \mathbb{T}^{\alpha,\mu},\\
&\qquad \Psi:=(I\oplus \mathcal{F})\big((\Omega^{T}\oplus \Omega^{T})\widetilde{\Omega}^{T}\big).\nonumber
\end{align}

\item 
Let the $H$-Hermitian normal form $\mathcal{A}$ with one eigenvalue $\lambda\in \mathbb{R}$ be:
\begin{align}
\mathcal{A}
=\bigoplus_{r=1}^{N}\Big( \bigoplus_{j=1}^{m_r}  \big( J_{\alpha_r}(\lambda)\big)\Big),\qquad 
H=\bigoplus_{r=1}^{N}\Big( \bigoplus_{j=1}^{m_r}  
\big( \vartheta_{r,j} E_{\alpha_r} \big) \Big)
\end{align}

By proposition \ref{resAoXXA} (\ref{resAoXXA1}) yields the solution of the equation $\mathcal{A}Q=Q\mathcal{A}$ is of the form $Q=[Q_{rs}]_{r,s=1}^{N}$ with an $\mu_r$-by-$\mu_s$ block matrix $Q_{rs}$ whose blocks are $\alpha_r$-by-$\alpha_s$ rectangular Toeplitz
of the form (\ref{QTY}).
\normalsize

\quad
Next, we write $H=EB$, where $E$ is as in (\ref{BKU}) and $B:=\bigoplus_{r=1}^{N}\big(\bigoplus_{j=1}^{m_r} \vartheta_{r,j} I_{\alpha_r}\big)$, hence $H$-unitarity of $Q$ yields
%
%
%
%
\begin{align}\label{QTQ2cU}
 H =  &   Q^{*}  H  Q\nonumber\\
E B = &  E  Q^{*}(E B) Q\\
B = & ( E  Q^{*} E ) B Q.\nonumber
\end{align}
We conjugate (\ref{QTQ2cU}) with $\Omega$ from Lemma \ref{lemaP} and get:
\begin{align}\label{ortoD2U}
&\mathcal{D} =  \mathcal{F}\mathcal{Q}^{*}\mathcal{F}\mathcal{D}\mathcal{Q},\\
%
%
\mathcal{Q}:=\Omega^{T}Q\Omega, \qquad
&\mathcal{B}:=  
\Omega^{T} B \Omega=\bigoplus_{r=1}^{N} 
\big(
\bigoplus_{j=1}^{\alpha_r}
I_{p_r,q_r}
\big)
\qquad (\mathcal{F}= \Omega^{T}E\Omega).\nonumber
\end{align}

\quad
Due to a special form of $Q$ (described above), $\mathcal{Q}$ is of the form (\ref{0T0}), thus Lemma \ref{EqT} for $\star=*$, $\varepsilon=1$, $\mathcal{X}=\mathcal{Q}$ and $\mathbb{B}$ as in (\ref{ortoD2U}) provides the solution of (\ref{ortoD2U}).
\end{enumerate}
This completes the proof of the theorem.
\end{proof}

\begin{remark}
\label{Rfw}
The reduction of the system of equations (\ref{QK2}) to the block Toeplitz structure is substantially more involved in the present setting than in \cite{TSOC,TSOS}. Although in some cases the procedure leads directly to $\mathbb T^{\alpha,\mu}$, this reduction still requires a substantial amount of work; in the remaining cases, one is led to the block Toeplitz equation of Section~\ref{cereq}.
In the real case, an additional difficulty is that determining the intersection of the corresponding complex centralizer or isotropy group with the group of real nonsingular matrices is itself a nontrivial step.
\end{remark}

\section{Acknowledgement}

The research was supported by 
Slovenian Research and Innovation Agency (grant no. P1-0291).

\vspace{-2mm}

\end{document}